 \documentclass[11pt,letterpaper]{amsart}
\numberwithin{equation}{section}
\usepackage{amsmath,amsthm,amsfonts,amscd,eucal}

\usepackage{amssymb}
\def\ca{{\mathcal A}}
\def\cb{{\mathcal B}}

\def\cs{{\mathcal S}}

\def\ga{{\mathfrak A}}

\def\gam{{\mathfrak M}}

\def\gx{{\mathfrak X}}

\def\bc{{\mathbb C}}

\def\bm{{\mathbb M}}

\def\a{\alpha}

\def\g{\gamma}  \def\G{\Gamma}
\def\d{\delta}

\def\r{\rho}

\def\t{\tau}
\def\f{\varphi} \def\F{\Phi}
\def\th{\theta}
\def\om{\omega} \def\Om{\Omega}

\def\ots{\overline{\otimes}}

\def\id{\mathop{\rm id}}

\def\ker{\mathop{\rm Ker}}
\def\idd{{\bf 1}\!\!{\rm I}}

\newcommand{\ty}[1]{\mathop{\rm {#1}}}

\newtheorem{thm}{Theorem}

\newtheorem{cor}{Corollary}
\newtheorem{prop}{Proposition}

\newtheorem{defin}{Definition}
\newtheorem{rem}{Remark}
\theoremstyle{definition}

\begin{document}

\title[hilbert modules and reduction theory]
{On the reduction theory of $W^{*}$-algebras by Hilbert modules}
\author{Francesco Fidaleo}
\address{Francesco Fidaleo\\
Dipartimento di Matematica\\
Universit\`{a} di Roma ``Tor Vergata''\\
Via della Ricerca Scientifica 1, 00133 Roma, Italy.}
\email{{\tt fidaleo@axp.mat.uniroma2.it}}
\author{L\'aszl\'o Zsid\'o}
\address{L\'aszl\'o Zsid\'o\\
Dipartimento di Matematica\\
Universit\`{a} di Roma ``Tor Vergata''\\
Via della Ricerca Scientifica 1, 00133 Roma, Italy.}
\email{{\tt zsido@axp.mat.uniroma2.it}}

\begin{abstract}
\vskip 0.2cm

We deal with the reduction theory of a $W^*$-algebra $M$ along a $W^*$-subalgebra $Z$ of the centre of $M$.
This is done by using Hilbert modules naturally constructed by suitable spatial representations of the abelian
$W^*$-algebra $Z$. We start with an exhaustive investigation of such kind of Hilbert modules, which is also
of self-contained interest. After explaining the notion of the reduction in this framework, we exhibit the reduction
of the standard form of a $W^*$-algebra $M$ along any $W^*$-subalgebra of its centre, containing the unit
of $M$. In a forthcoming paper, this result is applied to study the structure of the standard representation
of the $W^*$-tensor product $M_1\ots_Z M_2$ of two $W^*$-algebras $M_1$ and $M_2$ over a common
$W^*$-subalgebra $Z$ of the centres.

{\bf Acknowledgement}: The first named author acknowledges the MIUR Excellence Department project awarded
to the Dept. of Math., Univ. of Rome "Tor Vergata", CUP E83C23000330006, as well as INdAM-GNAMPA.
\end{abstract}

\maketitle

\section{Introduction}

We start with some notations frequently used in the sequel. For such a purpose,
we shall use the terminology of \cite{SZ1}. In particular,
\begin{itemize}
\item $(\, {\bf\cdot}\, |\, {\bf\cdot}\, )$ will denote the inner product of a (complex) Hilbert space and it will
be assumed linear in the first variable and antilinear in the second variable;
\item $\cb(H)$ will denote the $C^*$-algebra of all bounded linear operators on the Hilbert space $H$,
with the identity simply denoted by $\idd_H$;

\noindent If $H_1\;\!, H_2$ are two Hilbert spaces, then $\cb (H_1\;\! ,H_2)$ will stay for the
Banach space of all bounded linear operators $H_1\longrightarrow H_2\;\!$.

\item $\idd_M$ and $Z(M)$ will denote the identity and centre of a $W^*$-algebra $M$, respectively;
\item the $w$-{\it topology} on a $W^*$-algebra $M$ is the weak${}^*$topology, while
the $s$-{\it topology} on $M$ is the locally convex vector space topology generated by the family
of seminorms
\smallskip

\centerline{$p_\f(x):=\f(x^*x)^{1/2}\;\! ,\quad x\in M\;\! ,$}
\smallskip

\noindent where $\f$ runs over all normal positive forms on $M\;\! ;$
\item a {\it Maximal Abelian Sub-Algebra} (MASA for short) $A\subset M$ of a $W^*$-algebra $M$ is a subset of $M$ satisfying $A=A'\bigcap M$;
\item a {\it projection} $p$ belonging to a $W^*$-algebra $M$, is understood as a selfadjoint one
({\it i.e.} $p^*p=p$);
\item $l_M(x)$ and $r_M(x)$ will stay for the {\it left and right support-projection of an operator} $x$
in some $W^*$-algebra $M$, where $l_M(x)=r_M(x)=:s_M(x)$ if $x$ is normal,
$z_M(x)$ for the {\it central support-projection} of $x\in M$, and $s_M(\varphi)$ for the {\it support-projection
of a normal positive linear functional} $\varphi$ on $M\;\!$;
\item a projection $e$ in a $W^*$-algebra $M$ is called of {\it countable type} if 
any family of mutually orthogonal non-zero projections in $M$, majorised by $e\;\!$,
is at most countable; if $\idd_M$ is of countable type then we say that the $W^*$-algebra $M$
is of countable type;
\item a projection $e$ in a $W^*$-algebra $M$ is called {\it piecewise of countable type}
if there exists a family $(p_{\iota})_{\iota}$ of projections in $Z(M)$
with $\sum\limits_{\iota}p_{\iota}=\idd_M$ such that $e p_{\iota}$ is of countable type for
each $\iota\;\!$.
\item a {\it spatial representation} $(\pi,H_\pi)$ of the $W^*$-algebra $M$ on a Hilbert space
$H_\pi$ is a unital, faithful, normal $*$-homomorphism $\pi:M\longrightarrow \cb (H_\pi)\;\!$;
\item a {\it conjugation} on a Hilbert space $H$ is an involutive, isometrical, antilinear map $H\longrightarrow H$;
\item a von Neumann algebra $M\subset\cb(H)$ is said to be in {\it standard form} if there exists a
conjugation $J:H\longrightarrow H$ such that the mapping $x\longmapsto Jx^*J$ is a $*$-anti-isomorphism of $M$
onto its commutant $M'$, acting identically on the centre;
\item a {\it standard representation} of a $W^*$-algebra is a spatial representation whose image is a
von Neumann algebra in standard form; each $W^*$-algebra admits a standard representation, which
is unique up to unitary equivalence (e.g. \cite{SZ1}, 10.14 and 10.15).
\end{itemize}

The classical reduction theory of von Neumann (\cite{vN}, see also Part II of \cite{D1}, Chapter 3 of \cite{Sa},
\cite{Z}) is a tool to describe general von Neumann algebras on separable Hilbert spaces in terms of
factors, that is von Neumann algebras with trivial centre. Thus it serves to reduce work with general
von Neumann algebras on separable Hilbert spaces to work with factors.

More precisely, in the formulation of \cite{Sa}, Chapter 3 or \cite{Z}, any von Neumann algebra $M$
on a separable Hilbert space is embedded in some $L^\infty\big(\mu\, ,\cb(K)\big)\;\!$, where $\mu$
is a finite Borel measure on a compact metric space $\Omega$ and $K$ is a separable Hilbert
space, such that the centre $Z(M)$ of $M$ corresponds to the scalar valued functions in
$L^\infty\big(\mu\, ,\bc\big) \equiv L^\infty\big(\mu )\;\!$.

\noindent By such an embedding, the image of $M$ turns out to be
\smallskip

\centerline{$\big\{ F\in L^\infty\big(\mu\, ,\cb(K)\big) ; F(\omega )\in M_{\omega}\text{ $\mu$-almost
everywhere}\big\}$}
\smallskip

\noindent with $M_{\omega}\subset \cb(K)\;\! ,\omega\in\Omega\;\! ,$ a family of factors, and the
commutant of the above image is
\smallskip

\centerline{$\big\{ F\in L^\infty\big(\mu\, ,\cb(K)\big) ; F(\omega )\in (M_{\omega})'\text{ $\mu$-almost
everywhere}\big\}\;\! .$}
\smallskip

Noticing that $L^\infty\big(\mu\, ,\cb(K)\big)$ is a type $\ty{I}$ von Neumann algebra with
centre $L^\infty\big(\mu )$ and it acts by $L^\infty\big(\mu )$-module maps on the Hilbert
$L^\infty\big(\mu )$-module $L^2\big(\mu\, ,K\big)$, it is natural to try to overcome the
separability restriction of von Neuman's reduction theory by looking to embed any
von Neumann algebra $M$ in a type $\ty{I}$ von Neumann algebra $N$ such that
the centre $Z(M)$ goes onto $Z(N)\;\!$. Since, with $e$ an abelian projection in $N$
of central support $\idd_N\;\!$, $N$ is acting by the left multiplication operators on the
Hilbert $Z(N)$-module $Ne$ and the left multiplication operators are $Z(N)$-module
maps, identifying $M$ with its image in $N$, one obtains a representation of $M$ as an
algebra of $Z(M)$-module maps on a Hilbert $Z(M)$-module.
Such a representation could replace the reduction theory of von Neumann by working with
general von Neumann algebras.
 
This strategy was used in \cite{SZ} to prove the commutation theorem
for tensor products over von Neumann subalgebras by transposing the proof
given in \cite{RV-D} for ordinary tensor products to the Hilbert module setting.
For this purpose, the geometry of the Hilbert $Z(N)$-module $Ne$, with $N$ a type $\ty{I}$ von
Neumann algebra and $e\in N$ an abelian projection of central support $\idd_N\;\!$, was
thoroughly discussed and then von Neumann subalgebras $M\subset N$ containing $Z(N)\;\!$,
hence satisfying $Z(N)\subset Z(M)$ were investigated.

In this paper, after shortly commenting on the idea of reduction by Hilbert module representations in
Section \ref{sec2} and exposing a summary of the general theory of Hilbert modules in Section \ref{sec3},
in Section \ref{sec4} we perform a systematic study of a class of Hilbert modules, naturally associated
to a type $\ty{I}$ $W^*$-algebra $N$ and an abelian projection in $N$of central support $\idd_N\;\!$,
which are basic for the proposed reduction theory. In particular, we give a complete description
of the structure of these Hilbert modules.

In Section \ref{sec5} we formulate explicitely the reduction by Hilbert module representations
and we comment the nature of the proposed reduction.
Finally, in Section \ref{sec6} we show how standardness of a von Neumann algebra $M$ can be encoded
in its reduction along a von Neumann subalgebra $\idd_M\in Z\subset Z(M)\;\!$.
Our main result is a characterization of standardness along $Z\;\!$: Theorem \ref{st1}.
It will be applied in \cite{FZ1} in the theory of tensor products of von Neumann algebras over
von Neumann subalgebras.
\smallskip

\section{The idea of reduction by Hilbert module representations}
\label{sec2}

A {\it reduction} of a von Neumann algebra $M\subset \cb(H)$ usually means, roughly
speaking, a decomposition of $M$ "along its centre" $Z(M)$ in a family of factors.

Representing the centre $Z(M)$ as $L^\infty(\mu )\;\!$, with $\mu$ a Radon measure
(regular Borel measure, finite on compact sets) on a locally compact topological space
$\Omega\;\!$, one can try to represent the elements of $M$ by functions in some
$L^\infty\big(\mu\, ,\cb(K)\big)$ such that, for an appropriate family of factors
$M_\omega\subset \cb(K)\, ,\omega\in\Om\,$, $F\in L^\infty\big(\mu\, ,\cb(K)\big)$
corresponds to an element of $M$ if and only if $F(\omega )\in M_\omega$ for almost
every $\omega\in\Om\,$. Thus we can identify $M$ with the algebra of all measurable
sections of an appropriate measurable fibration over $\Om\,$.

More generally, we can seek a decomposition of $M$ along a von Neumann subalgebra
$\idd_M\in Z\subset Z(M)$ of $Z(M)$. In fact, representing $Z$ as before by some $L^\infty(\mu )\;\!$,
we can try to identify $M$ with a subalgebra of an appropriate $L^\infty\big(\mu\, ,\cb(K)\big)$,
whose elements $F$ are described by the condition $F(\om)\in M_\om$ for almost every $\om\in\Om\,$,
where $M_\om\subset \cb(K)\;\! ,\om\in\Omega\;\!$, are von Neumann algebras, not necessarily
factors.

Such a decomposition yields also a decomposition of the centre in the family $Z(M_\om),\om\in\Om\;\!$,
which in the case $Z=Z(M)$ would reduce to $Z(M_\om)=\bc$ for almost every $\om\in\Om\;\!$.

In the other extreme case $Z=\bc$, we would have $\Om=\{\om_0\}$, $K=H$, $M_{\om_0}=M$.

This program can be satisfactorily completed only if $H$ is separable, in which
case also $K$ will be separable and $\Omega$ can be chosen metrizable
(John von Neumann 1949, \cite{vN}).

On the other hand, if we represent the centre of $M$ as $C(\Omega )\;\!$, where $\Omega$ is
the Gelfand spectrum of the centre $Z(M)\;\!$, the hyperstonean
space of all unital algebra homomorphisms $\omega : Z(M)\longrightarrow\bc\;\!$,
then we can try to extend every $\omega\in\Omega$ to a $*$-homomorphism
$\pi_\omega : M\longrightarrow \widetilde{M_\omega}\,$, where
$\widetilde{M_\omega}$ is a factor, such that, for each $x\in M$, the elements
$ \pi_\omega (x)\in \widetilde{M_\omega}$ depend continuously on
$\omega\in\Omega$ in an appropriate sense. In this way, we can identify
$M$ with the continuous sections of a kind of fibration over $\Omega\,$. 

For this purpose, we have first to get a family of norm-closed two-sided ideals
$(\gam_\omega )_{\omega\in\Omega}$ of $M$ with
\medskip

\centerline{$\displaystyle \gam_\omega \cap Z(M) = \ker (\omega )\text{ for all }\omega\in\Omega\;\! ,$}
\medskip

\centerline{$\displaystyle \bigcap\limits_{\omega\in\Omega} \gam_\omega =\{ 0\}\;\! ,$}
\medskip

\noindent and then to complete each quotient $C^*$-algebra $M_\omega =
M/\gam_\omega$ in an appropriate way to a von Neumann algebra
$\widetilde{M_\omega}$ with trivial centre.

In the case of a finite von Neumann algebra $M$, we can do this by choosing
for $\gam_\omega$ the maximal ideal of $M$ with $\gam_\omega \cap Z(M)
= \ker (\omega )$, and it will follow that already $M_\omega$ is a finite factor.
Thus a further completion is not needed (F. B. Wright 1954, \cite{W}, J.Vesterstr\o m 1970, \cite{V}).

This program can be
carried out also in the case of a semifinite von Neumann algebra (\c{S}. Str\u{a}til\u{a}-L. Zs. 1973
\cite{SZ2, SZ3}), but the obtained representation
is difficult to work with.

If we set for $\gam_\omega$ the smallest two-sided ideal of $M$ containing
$\ker (\omega )$ (the so-called {\it Glimm Ideal}), then the functions
\[
\Omega\ni\omega\longmapsto \| x/\gam_\omega\|\, ,\quad x\in M
\]
are continuous (J. Glimm 1960 \cite{G}), so $(M_\omega )_{\omega\in\Omega}$ is, with todays
terminology, a continuous field of $C^*$-algebras. On the
other hand, the $C^*$-algebras $M_\omega\,$, even if they are not factors in
general, are primitive (H. Halpern 1969 \cite{H1, H2}). Thus this not fully
completed reduction theory can be used to treat many problems
concerning, for example, commutators, norms, numerical ranges
in von Neumann algebras. However, this kind of reduction is not appropriate to manage problems
concerning spatial representations, as for example, the standard form.

Let us come back to the reduction theory of von Neumann, in which a von Neumann algebra
$\idd_H\in M\subset \cb(H)$ on a separable Hilbert space is embedded in a type $\ty{I}$
von Neumann algebra of the form $L^\infty\big(\mu\, ,\cb(K)\big)$ with $K$ separable, which of course acts by
pointwise left multiplication on the Hilbert space $L^2(\mu\, ,K)$. We point out that $L^2(\mu\, ,K)$
is actually a Hilbert module over the centre $Z(M)\approx  L^\infty(\mu)$ of $M$, and
$L^\infty\big(\mu\, ,\cb(K)\big)$ acts by module maps on $L^2(\mu\, ,K)$.
In order to arrive to this structure, we have:
\begin{itemize}
\item[(i)] to consider the type $\ty{I}$ von Neumann algebra 
$$
N:=Z(M)'\supset M, M'\,;
$$
\item[(ii)] to consider an abelian projection $e\in N$ of central support $\idd_N$;
\item[(iii)] to represent $N$ as the von Neumann tensor product of $eNe\approx Z(M)$
with a type $\ty{I}$ factor $\cb(K)$ (assuming here of course that $N$ is homogeneous);
\item[(iv)] after considering a representation of $eNe$ as $L^\infty(\mu)$, to identify
$N\approx eNe\overline{\otimes}\cb(K)$ with
$L^\infty(\mu)\overline{\otimes}\cb(K)=L^\infty\big(\mu\, ,\cb(K)\big)$, which acts on
$L^2(\mu\, ,K)$ by pointwise left multiplication.
\end{itemize}

To overcome the separability restriction, required by measure theoretical reasons, we skip
the identification of $eNe\approx Z(M)$ with $L^\infty(\mu)$. Thus we keep only the picture as follows
(cf. \cite{SZ}, comments before Section 1.1).

First, we consider on $Ne$ the left Hilbert module structure over $Z(M)$ given by the module
inner product

\centerline{$(xe\!\mid\!ye)_{\substack{ {} \\ Ne}}=z\in Z(M)\;\! ,$}
\medskip

\noindent where $z$ is uniquely defined by $ze=ey^*xe$.
Second, let $N$ act on $Ne$ by the module maps given by left multiplication. Here the reduction of
$M\subset N$ consists in considering its left action by module maps over $Z(M)$ on $Ne$.

Similarly as in the von Neumann approach, we can start with any von Neumann subalgebra
$\idd_M\in Z\subset Z(M)$. Then, for an abelian projection  $e\subset N:=Z'$ of central support
$\idd_N$ at our choice, we can follow the above construction.

We notice that also $M'\subset N$ will be reduced simultaneously with $M$, as in the von Neumann
approach.

\section{Preliminaries on  Hilbert modules}
\label{sec3}

Let us recall some basic notions on Hilbert
modules (see \cite{K1, Pa} and Chapter 15 of \cite{W-O}).

Let $B$ be a $C^{*}$-algebra, and $X$ a left $B$-module. Let further
$$
(\,{\bf\cdot}\mid{\bf\cdot}\,)_{X}: X\times X\to B
$$
be a conjugate bilinear map, linear in the first variable and antilinear in the second one, satisfying the
following conditions:
\begin{itemize}
\item[(i)] $(bx\!\mid\!y)_{X}=b(x\!\mid\!y)_{X}$, and\\
$\quad(x\!\mid\!by)_{X}=(x\!\mid\!y)_{X}b^*\,,\,\, x,y\in X\,, b\in B$;
\item[(ii)] $(x\!\mid\!x)_{X}\geq0\,,\quad x\in X$;
\item[(iii)] $(x\!\mid\!x)_{X}=0$ only if $x=0$.
\end{itemize}
By the polarisation formula, (ii) implies also
\begin{itemize}
\item[(iv)] $(x\!\mid\!y)_{X}=(y\!\mid\!x)_{X}^*\,,\,\, x,y\in X$.
\end{itemize}
Putting
$$
\|x\|_{X}:=\|(x\!\mid\!x)_{X}\|^{1/2}\,,\quad x\in X\;\! ,
$$
one obtains a norm $\|x\|_{X}$ on $X$ satisfying the Schwarz inequality
$$
\|(x\!\mid\!y)_{X}\|\leq\|x\|_{X}\|y\|_{X}\,,\quad x,y\in X
$$
(see \cite{Pa}, Prop. 2.3). If this norm is complete, the $B$-module $X$ equipped with the $B$-valued
inner product $(\,{\bf\cdot}\mid{\bf\cdot}\,)_{X}$, is said to be a {\it Hilbert $B$-module}.

Let $X$, $Y$ be Hilbert $B$-modules. A bounded linear (antilinear) map
$T:X\to Y$ is said to be adjointable if there exists another, necessarily 
bounded, linear (antilinear) map $T^{*}:Y\to X$ such that
$(Tx\!\mid\!y)_{Y}=(x\!\mid\!T^*y)_{X}$
($(Tx\!\mid\!y)_{Y}=(T^*y\!\mid\!x)_{X}$), $x,y\in X$.
It is immediate to verify that a linear (antilinear) adjointable map, as well as its adjoint, 
is automatically $B$-linear ($B$-antilinear).

Let $X_i$, $i=1,2,3$ be Hilbert $B$-modules and 
$$
T:X_1\to X_2\,,\quad S:X_2\to X_3
$$
are bounded $B$-linear ($B$-antilinear) maps, then the composition $ST:X_1\to X_3$ will be a
bounded $B$-linear or $B$-antilinear map (depending on the nature of $S$ and $T$) satisfying
$$
\|ST\|\leq\|S\|\|T\|\,.
$$
If $S$ and $T$ are adjointable, bounded, both linear or antilinear maps, then the composition
$ST:X_1\to X_3$ is still adjointable with
$$
(ST)^*=T^*S^*\,.
$$

We denote the set of all the bounded $B$-module maps (i.e bounded $B$-linear maps) from a
Hilbert $B$-module $X$ into another Hilbert $B$-module $Y$
by $\cb_{B}(X,Y)$. Clearly, $\cb_{B}(X,Y)$ is complete w.r.t. the operator norm, that is a Banach
space. The subset of $\cb_{B}(X,Y)$ consisting of all the adjointable linear maps from $X$ into
$Y$ will be denoted by $\ca_{B}(X,Y)$. A straightforward verification leads to
\begin{equation}
\label{cstnr}
\|T^*T\|=\|T\|^2\,,\quad T\in \ca_{B}(X,Y)\,,
\end{equation}
and in particular
$$
\|T\|=\|T^*\|\,,\quad T\in \ca_{B}(X,Y)\,.
$$

A linear (antilinear) adjointable map $T:X\to Y$ is said to be
{\it $B$-unitary} ({\it $B$-antiunitary}) if
\begin{equation}
\label{uni}
T^{*}T=\id{}_{X}\,,\quad TT^{*}=\id{}_{Y}\,.
\end{equation}

A particular case of interest is when $X=Y$. As customary, we set
$\cb_{B}(X):=\cb_{B}(X,X)$ and $\ca_{B}(X):=\ca_{B}(X,X)$. Clearly,
$\cb_{B}(X)$ is a unital Banach algebra w.r.t. the operator norm, and
$\ca_{B}(X)$ is a unital subalgebra of $\cb_{B}(X)\;\!$, stable under the
involution $T\to T^*$.
By \eqref{cstnr}, $\ca_{B}(X)$ endowed with the above 
$*$-operation, becomes a unital $C^*$-algebra (see \cite{Pa}, Rem. 2.5).

We notice that $B$ itself is a left Hilbert $B$-module w.r.t. the $B$-valued inner product
$$
(a\!\mid\!b)_{B}:=ab^*\,,\quad a,b\in B\,.
$$

A Hilbert $B$-module $X$ is called {\it self-dual} if any $B$-linear ($B$-antilinear) map
$T:X\to B$ is of the form
\begin{equation}
\label{linant}
Tx=(x\!\mid\!x_T)_{B}\,,\,\,
\big(Tx=(x_T\!\mid\!x)_{B}\big)\,,\quad x\in X
\end{equation}
for some, necessarily unique, $x_T\in X$.
It is easy to check that any $B$-linear ($B$-antilinear) map $T:X\to B$ of the form \eqref{linant} is abjointable.
The reverse holds true when $B$ is unital. Indeed, if the $B$-linear ($B$-antilinear) map $T:X\to B$ is
adjointable, \eqref{linant} holds with $x_T=T^*\idd_B$.
For non unital $C^*$-algebra $B$, and Hilbert $B$-module $X\equiv B$, the map $T:X\to B$ given by $Tx=x$
is elementarwise adjointable but it cannot be expressed as in \eqref{linant}.

\section{Hilbert modules associated to type $\ty{I}$ $W^*$-algebras}
\label{sec4}

For reduction of von Neumann algebras by Hilbert module representations, we shall
use Hilbert modules of the following type. Let $N$ be a type $\ty{I}$ $W^*$-algebra,
and $e\in N$ an abelian projection with central support $z_{N}(e)=\idd_N\;\!$.
Denote $Z:=Z(N)$. Then the relation
$$
exe=\F_e(x)e\,,\quad x\in N\,,
$$
defines a normal conditional expectation $\F_e:N\longrightarrow Z$ of support $e\;\!$. In such a way,
$Ne$ becomes a Hilbert $Z$-module with the $Z$-valued inner product
$$
(x\!\mid\!y)_{\substack{ {} \\ Ne}}:=\F_e(y^*x)\,,\quad x,y\in Ne\, .
$$
We notice that, after identifying $Z$ with $Ze$ by the $*$-isomorphism
$$
Z\ni z\longmapsto ze\in Ze\,,
$$
the above inner product is simply given by $(x\!\mid\!y)_{\substack{ {} \\ Ne}}\equiv y^*x$.
By this identification
\begin{equation}
\label{op-norm}
\|x\|_{\substack{ {} \\ Ne}}=\big\| (x\!\mid\!x)_{\substack{ {} \\ Ne}}\big\|^{1/2}\equiv\|x^*x\|^{1/2}
=\|x\|\, ,\quad x\in Ne\, .
\end{equation}
It will be convenient to use also the notation
\begin{equation*}
|x|_{\substack{ {} \\ Ne}}:= (x\!\mid\!x)_{\substack{ {} \\ Ne}}^{\;\! 1/2}= \F_e(x^*x)\in Z^+\, ,\quad x\in Ne\, ,
\end{equation*}
so that $\|x\|_{\substack{ {} \\ Ne}}=\|\;\! |x|_{\substack{ {} \\ Ne}} \|$ for all $x\in Ne\;\! .$
We recall (see \cite{SZ}, Section 1.1, page 297) :
\begin{equation*}
\begin{split}
| (x\!\mid\!y)_{\substack{ {} \\ Ne}} |\leq\;& |x|_{\substack{ {} \\ Ne}} |y|_{\substack{ {} \\ Ne}}\,
(\text{Schwarz inequality})\;\! ,\hspace{14 mm}x\;\! ,y\in Ne\;\! , \\
|x+y|_{\substack{ {} \\ Ne}}\leq\;& |x|_{\substack{ {} \\ Ne}} + |y|_{\substack{ {} \\ Ne}}\,
(\text{Minkowski inequality})\;\! ,\quad x\;\! ,y\in Ne\;\! .
\end{split}
\end{equation*}

\noindent We recall also that a net $(x_{\iota})_{\iota}$ in $Ne$ is $s$-convergent to $x\in Ne$
if and only if $| x_{\iota}-x |_{\substack{ {} \\ Ne}}\xrightarrow{\;\; \iota\;\,} 0$ in the $s$-topology,
in which case also $x_{\iota}^{\;\! *}\xrightarrow{\;\; \iota\;\,} x^*$ in the $s$-topology, hence
$x_{\iota}\xrightarrow{\;\; \iota\;\,} x$ in the $s^*$-topology (see \cite{SZ}, Section 1.1, page 298).
\smallskip

Let us report some particularities of this kind of Hilbert modules, starting with
the following result.

\begin{thm}
\label{nopr1}
Let $N$ be a type $\ty{I}$ $W^*$-algebra with centre $Z$.
\begin{itemize}
\item[(i)]  If $e\in N$ is an abelian projection with central support $z_N(e)=\idd_N$, and $\gx\subset Ne$ is a $s$-closed $Z$-submodule, then there exists a unique
projection $p_\gx\in N$ such that $\gx =p_\gx Ne$.
\item[(ii)] If $p$ is a projection in $N$,
and $f\in pNp$ an abelian projection with central support $z_{pNp}(f)=p\;\!$,
then there exists an abelian projection $e\in N$ with $z_N(e)=\idd_N$
such that

\centerline{$ep=f=pe\, ,\; z_{N}(f)=z_N(p)\, ,\; z_{N}(f)z_{N}(e-f)=0\;\! .$}
\smallskip

\noindent Consequently, $pxe=(pxp)f$ for each $x\in N$, hence
$$
(pNp)f=pNe\subset Ne
$$
is a $s$-closed $Z$-submodule of $Ne$.
\end{itemize}
\end{thm}
\begin{proof}
(i) is \cite{SZ}, Lemma 1.6.

For (ii), we start by noticing that we have
\begin{equation}
\label{centr.supp}
z_{pNp}(f)=p\Longleftrightarrow z_N(f)\geq p\Longleftrightarrow
z_N(f)=z_N(p)\;\! .
\end{equation}
The last equivalence is obvious.
Next, if $z_{pNp}(f)=p\;\!$, then
\smallskip

\centerline{$\underbrace{f\big( \idd_N -z_N(f)\big)}_{=\, 0}p
=0$ and $\big( \idd_N -z_N(f)\big) p\in Z(N)p=Z(pNp)$}
\smallskip

\noindent imply that $\big( \idd_N -z_N(f)\big) p=0\;\!$, hence $p=z_N(f)p\leq
z_N(f)\;\!$. Finally, if we assume that $z_N(f)\geq p$ and choose some
projection $q\in Z(N)$ such that $z_{pNp}(f)\in Z(pNp)=Z(N)p$ is equal to $q\;\! p\;\!$,
then
\medskip

\centerline{$Z(N)\ni q\geq q\;\! p=z_{pNp}(f)\geq f$}
\smallskip

\noindent implies $q\geq z_N(f)\geq p\;\!$, so $z_{pNp}(f)=q\;\! p=p\;\!$.
\smallskip

Let now $f_1$ be an abelian projection in $N\big(\idd_N-z_{N}(f)\big)$ of central
support $z_{N(\idd_N-z_{N}(f))}(f_1)=\idd_N-z_{N}(f)$, and set
$e:=f+f_1\;\!$.
Since $f$ and $f_1$ are abelian projections in $N$ and
have orthogonal central supports of
sum $\idd_N$ in $N$, their sum $e$ is an abelian projection in $N$ of central
support $z_N(e)=\idd_N\;\!$.

By (\ref{centr.supp}) we have $\idd_N-z_{N}(f)\leq \idd_N-p\;\!$, so
$\big(\idd_N-z_{N}(f)\big) p=0$ and, consequently,
\smallskip

\centerline{$ep=(f+f_1)p=f+f_1\big(\idd_N-z_{N}(f)\big) p=f\, ,$}
\centerline{$pe=(ep)^*=f^*=f\, .$}
\smallskip

\noindent Therefore $ep=f=pe\;\!$. 

On the other hand, again by \eqref{centr.supp},
$z_N(f)=z_N(p)$ and it is clearly orthogonal to
\smallskip

\centerline{$z_{N}(e-f)=z_{N}(f_1)=\idd_N-z_{N}(f)\;\! .$}
\smallskip

Finally, let $x\in N$ be arbitrary.
Since $p\leq z_N(p)=z_N(f)\in Z(N)\;\!$, we have
\smallskip

\centerline{$pxe=pz_N(f)xe=pxz_N(f)e\;\! .$}
\smallskip

\noindent Taking into account that $z_{N}(f)z_{N}(e-f)=0$, we have also
$$
z_{N}(f)e=z_{N}(f)(f+(e-f))=f+z_{N}(f)z_{N}(e-f)=f\,,
$$
hence
\smallskip

\centerline{$pxe=pxz_N(f)e=pxf=(pxp)f\;\! .$}
\end{proof}

The next theorem, which reproduces some automatic continuity and adjointability
results from \cite{SZ}, is the heart of the Hilbert module approach of the
reduction theory of $W^*$-algebras.


To start with, let us consider, for $j=1,2$, type $\ty{I}$ $W^*$-algebras $N_j$ with centres $Z_j$,
and $e_j\in N_j$ abelian projections with central supports $z_{N_j}(e_j)=\idd_{N_j}$.
We also assume the existence of an abelian $W^*$-algebra $Z$ such that there are
$*$-isomorphisms 
$$
\pi_j:Z\to\pi_j(Z)=Z_j\subset N_j\,,\quad j=1,2\,.
$$
Then, $\pi_j(Z)$, $N_j$ and $N_je_j$, $j=1,2$, become all (left) $Z$-modules
in a natural way. 
Moreover, $N_je_j$ becomes a self-dual Hilbert $Z$-module under the $Z$-valued inner product defined by
$$
\langle x\!\mid\!y\rangle_{\substack{ {} \\ N_je_j}}:=\pi_j^{-1}\big((x\!\mid\!y)_{\substack{ {} \\ N_je_j}}\big)
=\big(\pi_j^{-1}\circ\F_{e_j}\big)(y^*x)\,,\quad x,y\in N_je_j\,.
$$

\begin{thm}
\label{mmainn}
For $j=1,2$, let $N_j$, $e_j$, $Z$, $\pi_j$ as above. 
\begin{itemize} 
\item[(i)] If $\gx\subset N_1e_1$ is a $Z$-submodule and $T:\gx\longrightarrow N_2$ is a
bounded $Z$-linear ($Z$-antilinear) map, then $T$ extends to a unique $s$-continuous
$Z$-linear ($Z$-antilinear) map
$\widetilde{T}:\overline{\gx}^s\longrightarrow N_2$. In addition, $\|\widetilde{T}\|=\|T\|$.
\item[(ii)] If $\gx_j\subset N_je_j$, $j=1,2$, are  $s$-closed $Z$-submodules and $T:\gx_1\longrightarrow \gx_2$ is a
bounded $Z$-linear $(Z$-antilinear$)$ map, then $T$ is adjointable.
\end{itemize}
\end{thm}
\begin{proof}
(i) is a consequence of \cite{SZ}, Lemma 1.3.

(ii) is essentially  \cite{SZ}, 1.13 (1). We sketch its proof for the convenience of the reader.

By Theorem \ref{nopr1}, (i), there exists projection $p_1:=p_\gx\in N_1$ such that $\gx_1=p_1N_1e_1$. Let us consider the
bounded $Z$-linear ($Z$-antilinear) map $T_0:N_1e_1\to N_2e_2$ defined by
$$
T_0x:=T(p_1x)\in\gx_2\subset N_2e_2\,,\quad x\in N_1e_1\,.
$$

According to  \cite{SZ}, Lemma 1.12 and the Remarks 1.13 (2), 1.13 (3), $T_0$ is adjointable. Therefore, with
$$
T^*y:=p_1T_0^*y\in p_1N_1e_1=\gx_1\,,\quad y\in \gx_2\,,
$$
we check that $T^*:\gx_2\to\gx_1$ is the adjoint of $T$. Indeed,
$T^*$ is a $Z$-linear ($Z$-antilinear) map satisfying
\begin{align*}
\langle Tx\!\mid\!y\rangle_{\substack{ {} \\ \gx_2}}
=&\langle Tx\!\mid\!y\rangle_{\substack{ {} \\ N_2e_2}}
=\langle T(p_1x)\!\mid\!y\rangle_{\substack{ {} \\ N_2e_2}}
=\langle T_0x\!\mid\!y\rangle_{\substack{ {} \\ N_2e_2}}
=\langle x\!\mid\!T_0^*y\rangle_{\substack{ {} \\ N_1e_1}}\\
=&\big(\pi_1^{-1}\circ\F_{e_1}\big)\big((T_0^*y)^*x\big)
=\big(\pi_1^{-1}\circ\F_{e_1}\big)\big((T_0^*y)^*p_1x\big)\\
=&\big(\pi_1^{-1}\circ\F_{e_1}\big)\big((p_1T_0^*y)^*x\big)
=\langle x\!\mid\! p_1T_0^*y\rangle_{\substack{ {} \\ N_1e_1}}\\
=&\langle x\!\mid\! T^*y\rangle_{\substack{ {} \\ N_1e_1}}
=\langle x\!\mid\! T^*y\rangle_{\substack{ {} \\ \gx_1}}
\end{align*}
$$
\big((Tx\!\mid\!y)_{\substack{ {} \\ \gx_2}}=(T^*y\!\mid\!x)_{\gx_1}\;\text{in the
antilinear situation}\big)\,,
$$
for every $x\in\gx_1$ and $y\in\gx_2$.

\end{proof}

We notice that, with $N_j$, $e_j$, $Z$, $\pi_j$, $j=1,2$, as in Theorem \ref{mmainn}, (ii) of this theorem entails that, for every s-closed $Z$-submodules $\gx_j\subset N_je_j$, we have
$\cb_Z(\gx_1,\gx_2)=\ca_Z(\gx_1,\gx_2)$.

In particular, for $N$ a type $\ty{I}$ $W^*$-algebra, $e\in N$ an abelian projection with central support $\idd_N$, and $\gx\subset Ne$ a s-closed $Z:=Z(N)$-submodule,
$\cb_Z(\gx)=\ca_Z(\gx)$ is a unital $C^*$-algebra.

In the situation described above, every $a\in N$ induces, by left multiplication, a bounded $Z$-linear map $L^{Ne}(a):Ne\to Ne$ :
$$
L^{Ne}(a)x:=ax\,,\quad x\in Ne\;\! .
$$
Clearly, 
$$
L^{Ne}:N\ni a\mapsto L^{Ne}(a)\in\cb_Z(Ne)
$$
is a $*$-homomorphism. Moreover, as a direct consequence of the above theorem, we have

\begin{thm}
\label{isotr} 
Let $N$ be a type $\ty{I}$ $W^*$-algebra with centre $Z$, and $e\in N$ an abelian projection with
central support $z_N(e)=\idd_N$. Then
$$L^{Ne}:N\ni a\mapsto L^{Ne}(a)\in\cb_Z(Ne)$$
is a $*$-isomorphism. In particular, $\cb_Z(Ne)$ is a type $\ty{I}$ $W^*$-algebra and
\begin{equation}
\label{op.norm}
\| a\| =\| L^{Ne}(a)\| =\sup \{ \| ax\|\;\! ; x\in Ne\;\! , \| x\|\leq 1\}\;\! ,\quad a\in N\;\! .
\end{equation}
\end{thm}
\begin{proof}
The injectivity of $L^{Ne}$ is straightforward. Indeed, if $L^{Ne}(a)=0$ then we have for every $u\in U(N)$,
$$
a^*a(ueu^*)=a^*\big(L^{Ne}(a)(ue)\big)u^*=0\, ,
$$
so the support $s(a^*a)$ is orthogonal to $ueu^*$. Consequently, $s(a^*a)$ is orthogonal to the central support
$$
\idd_N=z_N(e)=\bigvee_{u\in U(N)}ueu^*\, ,
$$
hence $a=0$.

Since $\cb_Z(Ne)=\ca_Z(Ne)$ is a unital $C^*$-algebra, for the surjectivity of $L^{Ne}$ it is enough to
prove that for each unitary $T\in \ca_Z(Ne)$ there exists an isometrical $a\in N$ such that
$$
Tx=ax\, ,\quad x\in Ne\, .
$$
By the $s$-continuity of $T$, we can reduce the matter to the case when $N$ is a homogeneous type
$I$ $W^*$-algebra.

Let $(e_\iota)_\iota$ be a family of pairwise orthogonal, equivalent, abelian projections of central support
$\idd_N$ in $N$ such that $\sum_\iota e_\iota=\idd_N$. Since each $e_\iota$ is equivalent to $e$, there
are partial isometries $(v_\iota)_\iota\subset N$ satisfying
$$
v^*_\iota v_\iota=e\,,\quad v_\iota v^*_\iota=e_\iota
$$
for all $\iota$.

Since $T$ is assumed unitary, by \eqref{uni} we obtain for every two indices $\iota_1$ and $\iota_2$,
\begin{align*}
(Tv_{\iota_1}\!\mid\!Tv_{\iota_2})_{\substack{ {} \\ Ne}}=&(v_{\iota_1}\!\mid\!T^*Tv_{\iota_2})_{\substack{ {} \\ Ne}}=(v_{\iota_1}\!\mid\!v_{\iota_2})_{\substack{ {} \\ Ne}}\\
=&\F_e(v_{\iota_2}^*v_{\iota_1})=\d_{\iota_1\iota_2}\idd_N\,.
\end{align*}
In other words, we have
$$
(Tv_{\iota_2})^*Tv_{\iota_1}=\d_{\iota_1\iota_2}e
$$
for all $\iota_1$ and $\iota_2$. Consequently, we can define
\begin{equation}
\label{iotaz}
a:=\sum_\iota(Tv_\iota)v_\iota^*\in N\,
\end{equation}
where the sum is $s$-convergent. Moreover,
\begin{align*}
a^*a=&\sum_{\iota_1,\iota_2}v_{\iota_2}(Tv_{\iota_2})^*(Tv_{\iota_1})v_{\iota_1}^*
=\sum_{\iota_1,\iota_2}v_{\iota_2}(\d_{\iota_1\iota_2}e)v_{\iota_1}^*\\
=&\sum_\iota v_\iota v_\iota^*=\sum_\iota e_\iota=\idd_N\,,
\end{align*}
that is $a$ is isometrical.

Now we verify that
$$
Tx=ax\,,\quad x\in Ne\,.
$$
Indeed, as
\begin{equation}
\label{iotazz}
x=\sum_\iota e_\iota x=\sum_\iota v_\iota v_\iota^*x=\sum_\iota(x\!\mid\!v_\iota)_{\substack{ {} \\ Ne}}v_\iota\,,
\end{equation}
by the $s$-continuity of $T$ we obtain
$$
Tx=\sum_\iota T\big((x\!\mid\!v_\iota)_{\substack{ {} \\ Ne}}v_\iota\big)=\sum_\iota (x\!\mid\!v_\iota)_{\substack{ {} \\ Ne}}Tv_\iota\,.
$$
On the other hand, by \eqref{iotaz} and \eqref{iotazz},
\begin{align*}
ax=&\sum_{\iota_1,\iota_2}(Tv_{\iota_1})v_{\iota_1}^*(x\!\mid\!v_{\iota_2})_{\substack{ {} \\ Ne}}v_{\iota_2}\\
=&\sum_{\iota_1,\iota_2}(x\!\mid\!v_{\iota_2})_{\substack{ {} \\ Ne}}(Tv_{\iota_1})(\d_{\iota_1\iota_2}e)\\
=&\sum_\iota(x\!\mid\!v_\iota)_{\substack{ {} \\ Ne}}Tv_\iota\,,
\end{align*}
that is $Tx=ax$.

\end{proof}

Theorem \ref{isotr} can be easily generalised to the case of $\cb_Z(\gx)$, where $N$ is a type $\ty{I}$ $W^*$-algebra with centre $Z$, $e\in N$ an abelian projection with central support $\idd_N$, and 
$\gx\subset Ne$ a s-closed $Z$-submodule. We recall that,
by Theorem \ref{nopr1}, (i), $\gx=p_\gx Ne$ for a unique projection $p_\gx \in N$. For every $a\in p_\gx Np_\gx$, we can define a bounded $Z$-linear map
$L^\gx(a):\gx\to\gx$ by setting
$$
L^\gx(a)x:=ax\,,\quad x\in\gx\,.
$$
\begin{cor}
\label{coun}
For $N$, $Z$, $e$, $\gx$ as above,
$$
L^\gx:p_\gx Np_\gx\ni a\mapsto L^\gx(a)\in\cb_Z(\gx)
$$
is a $*$-isomorphism. In particular, $\cb_Z(\gx)$ is a type $\ty{I}$ $W^*$-algebra.
\end{cor}
\begin{proof}
We start by noticing that $L^\gx$ is a $*$-homomorphism.

For the injectivity of $L^\gx$, let $a\in p_\gx Np_\gx$ be such that $L^\gx(a)=0$. Then we have, for every $x\in Ne$, 
$$
L^{Ne}(a)x=ax=(ap_\gx)x=a(p_\gx x)=L^{\gx}(a)(p_\gx x)=0\,,
$$
and Theorem \ref{isotr} yields $a=0$.

For the surjectivity of $L^\gx$, let $T\in \cb_Z(\gx)$ be arbitrary. Then 
$$
T_o:Ne\ni x\longmapsto T(p_\gx x)\in\gx\subset Ne
$$
is a bounded $Z$-linear map. By Theorem \ref{isotr}, there exists $a\in N$ such that $T_o=L^{Ne}(a)$. For every $x,y\in Ne$, we have
\begin{equation}
\label{stima}
(ax\!\mid\!y)_{\substack{ {} \\ Ne}}=\big(T(p_\gx x)\!\mid\!p_\gx y\big)_{\substack{ {} \\ \gx}}\;\! .
\end{equation}
Indeed,
\begin{align*}
(ax\!\mid\!y)_{\substack{ {} \\ Ne}}=&(L^{Ne}(a)x\!\mid\!y)_{\substack{ {} \\ Ne}}=(T_ox\!\mid\!y)_{\substack{ {} \\ Ne}}
=\big(\underbrace{T(p_\gx x)}_{\in\gx}\!\mid\!y\big)_{\substack{ {} \\ Ne}}\\
=&\big(p_\gx T(p_\gx x)\!\mid\!y\big)_{\substack{ {} \\ Ne}}=\big(T(p_\gx x)\!\mid\!p_\gx y\big)_{\substack{ {} \\ Ne}}
=\big(T(p_\gx x)\!\mid\!p_\gx y\big)_{\substack{ {} \\ \gx}}\,.
\end{align*}
In \eqref{stima}, replacing $x,y$ with $p_\gx x,p_\gx y$, respectively, we obtain
\begin{equation}
\label{stina1}
(ap_\gx x\!\mid\!p_\gx y)_{\substack{ {} \\ Ne}}=\big(T(p_\gx x)\!\mid\!p_\gx y\big)_{\substack{ {} \\ \gx}}\,,
\end{equation}
hence
\begin{align*}
(L^{Ne}(p_\gx ap_\gx)x\!\mid\!y)_{\substack{ {} \\ Ne}}=&(p_\gx ap_\gx x\!\mid\!y)_{\substack{ {} \\ Ne}}=(ap_\gx x\!\mid\!p_\gx y)_{\substack{ {} \\ Ne}}\\
\overset{\text{\eqref{stina1}}}{=}&\big(T(p_\gx x)\!\mid\!p_\gx y\big)_{\substack{ {} \\ \gx}}\overset{\text{\eqref{stima}}}{=}(ax\!\mid\!y)_{\substack{ {} \\ Ne}}\\
&=(L^{Ne}(a)x\!\mid\!y)_{\substack{ {} \\ Ne}}\,.
\end{align*}
We conclude that $L^{Ne}(p_\gx ap_\gx)=L^{Ne}(a)$ and, by the injectivity of $L^{Ne}$, $a=p_\gx ap_\gx\in p_\gx Np_\gx$. Therefore,
$$
Tx=T_ox=L^{Ne}(a)x=ax=L^{\gx}(a)x
$$
for every $x\in\gx$, that is $T=L^{\gx}(a)$.

\end{proof}

For $N$, $Z$, $e$, $\gx$ as before, let us denote the
commutant of $\cs\subset \cb_Z(\gx)$, that is
$$
\cs'=\big\{T\in\cb_Z(\gx)\,;\,\, TS=ST\,,\,S\in\cs\big\}\,.
$$

Since $L^\gx$ is a $*$-isomorphism, we have for any set $\cs\subset p_\gx Np_\gx$,
$$
L^\gx(\cs)'=L^\gx(\cs)'\bigcap L^\gx\big(p_\gx Np_\gx\big)=L^\gx\big(\cs'\bigcap (p_\gx Np_\gx)\big)\,.
$$
Consequently, 
$$
L^\gx(\cs)''=L^\gx\Big(\big(\cs'\bigcap (p_\gx Np_\gx)\big)'\bigcap(p_\gx Np_\gx)\Big)\supset L^\gx(\cs)\,,
$$
and the equality $L^\gx(\cs)''=L^\gx(\cs)$ holds if and only if 
$$
\big(\cs'\bigcap (p_\gx Np_\gx)\big)'\bigcap(p_\gx Np_\gx)=\cs\,.
$$
This happens in the following situation:
\begin{cor}
\label{docomm1}
For $N$, $Z$, $e$, as before, $M\subset N$ a $W^*$-subalgebra of $N$ containing $Z$, and $p$ a projection either in $M$ or in $M'\bigcap N$, denoting $\gx:=pNe$, we have
$$
L^\gx(pMp)''=L^\gx(pMp)\,.
$$
\end{cor}
\begin{proof}
According to the above considerations, we have to prove the equality
\begin{equation}
\label{stina2}
\Big((p_\gx Mp_\gx)'\bigcap (p_\gx Np_\gx)\Big)'\bigcap(p_\gx Np_\gx)=p_\gx Mp_\gx\,.
\end{equation}

We recall that if $M_o\subset N_o$ a $W^*$-subalgebra of a type $\ty{I}$ $W^*$-algebra $N_o$ with $Z(N_o)\subset M_o$, the double relative commutant of $M_o$ in $N_o$ is equal to $M_o$ itself
(see e.g. \cite{D1}, Part III, Ch. 7, Ex. 13 b):
\begin{equation}
\label{stina3}
\big(M_o'\bigcap N_o\big)'\bigcap N_o=M_o\,.
\end{equation}

To end the proof, we notice that \eqref{stina3} can be applied with $N_o=p_\gx Np_\gx$ and $M_o=p_\gx Mp_\gx$, obtaining \eqref{stina2}.

\end{proof}

The generalisation of Theorem \ref{isotr} to the case of bounded module maps between different Hilbert modules is more involved. Starting to do this,
let now $Z$ be an abelian $W^*$-algebra and, for $j=1,2$, $N_j$ a type $\ty{I}$ $W^*$-algebra, $e_j\in N_j$ an abelian projection with central support $z_{N_j}(e_j)=\idd_{N_j}\;\!$, and 
$\pi_j:Z\to Z(N_j)$ a $*$-isomorphism.
Let us consider a spatial representation $(\r_j,H_j)$ of $N_j$ such that 
$\r_j(N_j)'=\r_j(Z(N_j))=\r_j(\pi_j(Z))$.

Setting
$$
\pi(z):=\r_1(\pi_1(z))\oplus\r_2(\pi_2(z))\,,\quad z\in Z\,,
$$
we denote
$$
N:=\pi(Z)'\,.
$$
Clearly, $N\subset\cb(H_1\bigoplus H_2)$ is a type $\ty{I}$ von Neumann algebra, and $Z\ni z\mapsto\pi(z)\in Z(N)$ is a $*$-isomorphism of $Z$ onto the centre of $N$.

Since $z_{N_j}(e_j)=\idd_{N_j}$,
$$
\th_j:Z\ni z\mapsto (\r_j\circ\pi_j)(z)\lceil_{\r_j(e_j)H_j}\in\cb(\r_j(e_j)H_j)
$$
is a $*$-isomorphism of $Z$ onto the induced von Neumann algebra

\noindent $\r_j(Z(N_j))_{\r_j(e_j)}\subset\cb(\r_j(e_j)H_j)$ which is a MASA. Consequently

\noindent $\th_2\circ\th_1^{-1}$ is a $*$-isomorphism of the MASA $\r_1(Z(N_1))_{\r_1(e_1)}$
onto the MASA $\r_2(Z(N_2))_{\r_j(e_2)}$.
According to \cite{SZ1}, Lemma 7.2, E.7.15 and Corollary 5.25, it is implemented by a unitary
$U:\r_1(e_1)H_1\to \r_2(e_2)H_2$:
$$
\th_2\circ\th_1^{-1}(T)=UTU^*\,,\quad T\in\r_1(Z(N_1))_{\r_1(e_1)}\,.
$$
Extending $U$ by $0$ on $H_1\ominus\r_1(e_1)H_1$, we obtain a partial isometry $V:H_1\to H_2$ satisfying 
\begin{equation}
\label{def.V}
V^*V=\r_1(e_1)\,,\quad VV^*=\r_2(e_2)\,,
\end{equation}
and we have
\begin{equation}
\label{V-intertw.}
V\r_1(\pi_1(z))=\r_2(\pi_2(z))V\,,\quad z\in Z\,.
\end{equation}

Let us denote, for $j\;\! ,k\in \{ 1\;\! , 2\}\;\!$,
\begin{equation}
\label{cucus}
\begin{split}
&(N_j ,\pi_j ; N_k ,\pi_k)\\
:=&\{ T\in \cb (H_j\;\! ,H_k)\;\! ;  T\r_j(\pi_j(z))
=\r_k(\pi_k(z))T , z\in Z\}\;\! .
\end{split}
\end{equation}
By (\ref{V-intertw.}), we have $V\in (N_1 ,\pi_1 ; N_2 ,\pi_2)\;\!$. On the other hand, by the

\noindent bicommutant theorem of von Neumann,
\begin{equation}
\label{bikkk}
(N_j ,\pi_j ; N_j ,\pi_j)= \r_j(\pi_j(Z))' =\r_j(N_j)\,,\quad j=1,2\,.
\end{equation}

\noindent Clearly, for any $j\;\! ,k\;\! ,l\in \{ 1\;\! , 2\}\;\!$,
\begin{equation}
\label{intertw.}
\begin{split}
(N_k ,\pi_k ; N_l ,\pi_l) (N_j ,\pi_j ; N_k &,\pi_k)\subset(N_j ,\pi_j ; N_l ,\pi_l)\,,\\
(N_j ,\pi_j ; N_k ,\pi_k)^*=&(N_k ,\pi_k ; N_j ,\pi_j)\,.
\end{split}
\end{equation}

Finally, let $e\in\cb(H_1\bigoplus H_2)$ be defined by the matrix-representation
$$
e:=\frac12\begin{pmatrix} \r_1(e_1)&V^* \\ V& \r_2(e_2) \end{pmatrix}\, .
$$
Notice that $e$ is a projection of $\cb (H_1\bigoplus H_2)$.
\smallskip

The handling of $\cb_Z (N_1e_1,N_2e_2)$ will be reduced to work with $\cb_Z(N e)$ by using
the next proposition.

\begin{prop}
\label{wuatt}
With the above notations, the following assertions hold true.
\begin{itemize}
\item[(i)] $T=\begin{pmatrix} 
	 T_{11}&T_{12}\\
	T_{21}&T_{22}\end{pmatrix}\in N\iff T_{jk}\in (N_k ,\pi_k ; N_j ,\pi_j)\;\! ,\; j\;\! ,k=1,2$.
\item[(ii)] $e$ is an abelian projection of $N$ with central support $z_N(e)=\idd_{H_1\bigoplus H_2}$.
\item[(iii)] $Ne =\left\{ \begin{pmatrix} \r_1(x_1)\quad\!& \r_1(x_1) V^*\\
\r_2(x_2) V&\r_2(x_2)\quad\, \end{pmatrix} ; x_1\in N_1e_1\;\! , x_2\in N_2e_2\right\}\;\! .$
\item[(iv)] For $j=1,2$, and $x_j\in N_je_j$, if $0\leq z_j\in Z$ is defined by $x_j^*x_j=\pi_j(z_j)e_j$,
then
\begin{equation}
\label{settete}
\begin{split}
\left\|\begin{pmatrix} 
	 \r_{1}(x_1)& \r_{1}(x_1)V^*\\
	 \r_{2}(x_2)V& \r_{2}(x_2)\end{pmatrix}\right\|
=&\sqrt2\:\!\big\|z_1+z_2\big\|^{1/2}\\
\geq&\sqrt2\max\big(\|x_1\|,\|x_2\|\big)\,.
\end{split}
\end{equation}
Moreover, if additionally $x_1=0$ or $x_2=0$, the inequality in \eqref{settete} is indeed an equality.
\end{itemize}
\end{prop}

\begin{proof}
(i) follows by using, for $T=\begin{pmatrix} 
	 T_{11}&T_{12}\\
	T_{21}&T_{22}\end{pmatrix}\in \cb (H_1\bigoplus H_2)$ and

\noindent $z\in Z$, the formulas
\medskip

\centerline{$\displaystyle T \pi (z)=\begin{pmatrix}  T_{11}\r_1(\pi_1(z))
&T_{12} \r_2(\pi_2(z)) \\ T_{21}\r_1(\pi_1(z)) &T_{22} \r_2(\pi_2(z)) \end{pmatrix} ,$}
\smallskip

\centerline{$\displaystyle \pi (z) T =\begin{pmatrix} \r_1(\pi_1(z)) T_{11}
& \r_1(\pi_1(z)) T_{12} \\ \r_2(\pi_2(z)) T_{21} & \r_2(\pi_2(z)) T_{22} \end{pmatrix}.$}
\medskip

For inclusion $\subset$ in (iii), let $\displaystyle T=\begin{pmatrix} 
T_{11}&T_{12} \\ T_{21}&T_{22}\end{pmatrix}\in N$ be arbitrary. Then
\begin{equation*}
\begin{split}
T e=\;& \frac 12 \begin{pmatrix} T_{11}&T_{12} \\ T_{21}&T_{22}\end{pmatrix}
\begin{pmatrix} \r_1(e_1)&V^* \\ V& \r_2(e_2) \end{pmatrix} \\
=\; & \frac 12  \begin{pmatrix} T_{11}\r_1(e_1) + T_{12} V &
T_{11} V^* + T_{12} \r_2(e_2) \\ T_{21}\r_1(e_1) + T_{22} V &
T_{21} V^* + T_{22} \r_2(e_2) \end{pmatrix}
\end{split}
\end{equation*}
By the just proved (i) and by (\ref{intertw.}), we have
\smallskip

\centerline{$T_{11}\r_1(e_1) + T_{12} V = (T_{11}+T_{12} V) \r_1(e_1)$}
\smallskip

\noindent with
\smallskip

\centerline{$T_{11}+T_{12} V\in (N_1 ,\pi_1 ; N_1 ,\pi_1) +
(N_2 ,\pi_2 ; N_1 ,\pi_1) (N_1 ,\pi_1 ; N_2 ,\pi_2)$}
\smallskip

\noindent\hspace{2.85 cm}$\subset (N_1 ,\pi_1 ; N_1 ,\pi_1) = \r_1(N_1)\;\! ,$
\smallskip

\noindent hence $T_{11}\r_1(e_1) + T_{12} V\in \r_1(N_1) \r_1(e_1) =\r_1(N_1 e_1)\;\! .$
In other words,
\medskip

\centerline{$T_{11}\r_1(e_1) + T_{12} V =\r_1(2\;\! x_1)$ for some $x_1\in N_1e_1\;\! .$}
\medskip

\noindent Then, by (\ref{def.V}),
\medskip

\centerline{$T_{11} V^* + T_{12}\;\! \r_2(e_2) =T_{11} \r_1(e_1)V^* + T_{12} V V^* =
\r_1(2\;\! x_1) V^* .$}
\smallskip

\noindent Similarly,
\smallskip

\centerline{$T_{21} V^* + T_{22}\;\! \r_2(e_2) =\r_2(2\;\! x_2)$ for some $x_2\in N_2e_2$}
\smallskip

\noindent and

\centerline{$T_{21}\r_1(e_1) + T_{22} V =\r_2(2\;\! x_2) V .$}
\smallskip

\noindent In conclusion,
\smallskip

\centerline{$\displaystyle T e = \frac 12 \begin{pmatrix} \r_1(2\;\! x_1)\quad\!& \r_1(2\;\! x_1) V^*\\
\r_2(2\;\! x_2) V&\r_2(2\;\! x_2)\quad\, \end{pmatrix} = \begin{pmatrix} \r_1(x_1)\quad\!& \r_1(x_1) V^*\\
\r_2(x_2) V&\r_2(x_2)\quad\, \end{pmatrix}$}
\medskip

\noindent with $x_1\in N_1e_1$ and $x_2\in N_2e_2\;\!$.
\smallskip

For the proof of inclusion $\supset$ in (iii), let $x_1\in N_1e_1$ and $x_2\in N_2e_2$ be arbitrary.
Then
\medskip

\centerline{$\displaystyle \begin{pmatrix} \r_1(x_1)\quad\!& \r_1(x_1) V^*\\
\r_2(x_2) V&\r_2(x_2)\quad\, \end{pmatrix} = \begin{pmatrix} \r_1(x_1)& 0\\
0&\r_2(x_2) \end{pmatrix} \begin{pmatrix} \r_1(e_1)&V^* \\ V& \r_2(e_2) \end{pmatrix}$}
\smallskip

\noindent\hspace{5.02 cm}$=\begin{pmatrix} \r_1(2\;\! x_1)& 0\\
0&\r_2(2\;\! x_2) \end{pmatrix} e$
\medskip

\noindent belongs to $Ne$ acording to (i).
\smallskip

To prove that $e$ is an abelian projection in $N$, we have to show that $eNe\subset \pi(Z) e\;\!$.
According to (iii), the generic element of $eNe$ is of the form
\begin{equation*}
x:=e  \begin{pmatrix} \r_1(x_1)\quad\!& \r_1(x_1) V^*\\
\r_2(x_2) V&\r_2(x_2)\quad\, \end{pmatrix}
\end{equation*}
with $x_1\in N_1e_1\;\! ,x_2\in N_2e_2\;\!$, so the proof will be done by showing that
the above $x$ belongs to $\pi(Z) e\;\!$.

Since $e_j$ is an abelian projection in $N_j\;\! , j=1\;\! ,2\;\!$, there exist $z_1\;\! ,z_2\in Z$
such that $e_j x_j =\pi_j(z_j) e_j\;\! , j=1\;\! ,2\;\!$. By (\ref{def.V}) and \eqref{V-intertw.}, we then have also
\begin{equation*}
\begin{split}
V\r_1(x_1) =\;&V\r_1(e_1x_1)=V\r_1(e_1\pi_1(z_1))
=V\r_1(\pi_1(z_1))\\
 =\;&\r_2(\pi_2(z_1))V
\end{split}
\end{equation*}
and, similarly,
\begin{equation*}
V^*\r_2(x_2) =\r_1(\pi_1(z_2)) V^* .
\end{equation*}
Therefore, $2\;\! x$ is equal to
\begin{equation*}
\begin{split}
&\begin{pmatrix} \r_1(e_1)&V^* \\ V& \r_2(e_2) \end{pmatrix}
\begin{pmatrix} \r_1(x_1)\quad\!& \r_1(x_1) V^*\\
\r_2(x_2) V&\r_2(x_2)\quad\, \end{pmatrix} \\
=&\begin{pmatrix} \r_1(e_1x_1)+V^*\r_2(x_2)V &\r_1(e_1x_1)V^* +V^*\r_2(x_2) \\
V\r_1(x_1)+\r_2(e_2x_2)V & V\r_1(x_1)V^*+\r_2(e_2x_2)\end{pmatrix} \\
=&\begin{pmatrix} \r_1(\pi_1(z_1))\r_1(e_1)+\r_1(\pi_1(z_2)) V^*V\!\! &
\r_1(\pi_1(z_1))V^*+\r_1(\pi_1(z_2)) V^* \\
\r_2(\pi_2(z_1)) V +\r_2(\pi_2(z_2)) V &
\!\!\r_2(\pi_2(z_1)) VV^* +\r_2(\pi_2(z_2))\r_2(e_2) \end{pmatrix} \\
=&\begin{pmatrix} \r_1(\pi_1(z_1+z_2))\r_1(e_1) &\r_1(\pi_1(z_1+z_2))V^* \\
\r_2(\pi_2(z_1+z_2)) V & \r_2(\pi_2(z_1+z_2)) \r_2(e_2) \end{pmatrix} \\
=&\;2\;\!\pi (z_1+z_2)\;\! e\in \pi (Z) e\;\! .
\end{split}
\end{equation*}

It is easy to check that the central support of $e$ in $N$ is $\pi (\idd_Z)=\idd_N\;\!$. Indeed, if $p\in Z$ is a projection such that
$e\leq\pi(p)\iff e\pi(p)=e$, that is
$$
\frac12\begin{pmatrix} \r_1(e_1\pi_1(p)) &V^*\r_2(\pi_2(p)) \\
V\r_1(\pi_1(p)) & \r_2(e_2\pi_2(p))\end{pmatrix} 
=\frac12\begin{pmatrix} \r_1(e_1) &V^* \\
V& \r_2(e_2)\end{pmatrix}\,,
$$
then $e_1\pi_1(p)=e_1$. Since $z_{N_1}(e_1)=\idd_{N_1}$, $p$ must be equal to $\idd_Z$.

For (iv),

\medskip

\noindent\hspace{1.44 cm}$\left\|\begin{pmatrix} 
\r_{1}(x_1)& \r_{1}(x_1)V^*\\
	 \r_{2}(x_2)V& \r_{2}(x_2)\end{pmatrix}\right\|^2$
	 \smallskip

\noindent\hspace{1 cm}$= \left\|\begin{pmatrix} 
\r_{1}(x_1)& \r_{1}(x_1)V^*\\
	 \r_{2}(x_2)V& \r_{2}(x_2)\end{pmatrix}^{\!*}\!\! \begin{pmatrix} 
\r_{1}(x_1)& \r_{1}(x_1)V^*\\
	 \r_{2}(x_2)V& \r_{2}(x_2)\end{pmatrix}\right\|$
\smallskip

\noindent\hspace{1 cm}$= \left\|\begin{pmatrix} 
	 \r_{1}(x_1^*x_1)+V^*\r_{2}(x_2^*x_2)V& \r_{1}(x_1^*x_1)V^*+V^*\r_{2}(x_2^*x_2)\\
	 V \r_{1}(x_1^*x_1)+\r_{2}(x_2^*x_2)V& V \r_{1}(x_1^*x_1)V^*+\r_{2}(x_2)\end{pmatrix}\right\|$
\smallskip

\noindent\hspace{1 cm}$= \left\|\begin{pmatrix} 
	 \r_{1}(\pi_1(z_1+z_2)) \r_{1}(e_1)& \r_{1}(\pi_1(z_1+z_2)V^*\\
	 \r_{2}(\pi_2(z_1+z_2)V&  \r_{2}(\pi_2(z_1+z_2))\r_{2}(e_2)\end{pmatrix}\right\|$
\smallskip

\noindent\hspace{1 cm}$= 2\:\!\big\|\pi(z_1+z_2)\:\!e\big\|=2\:\!\big\|\pi(z_1+z_2)\big\|=2\:\!\|z_1+z_2\|$
\smallskip

\noindent\hspace{1 cm}$\geq 2\max\big(\|z_1\|,\|z_2\|\big)
=2\max\big(\|x_1\|^2,\|x_2\|^2\big)\;\! .$

\medskip

\noindent Finally, if for example $x_1=0$, and consequently $z_1=0$, then 
\begin{align*}
\|z_1+z_2\|=\|z_2\|=\|x_2\|^2
=\max\big(\|x_1\|^2,\|x_2\|^2\big)\,.
\end{align*}
\end{proof}

Using the above notations, each $T\in (N_1 ,\pi_1 ; N_2 ,\pi_2)$ induces, by left
multiplication, a bounded $Z$-linear map $\mathcal{L}^{N_1e_1,N_2e_2}(T) : N_1e_1\to N_2e_2\;\!$.
More precisely, for every $x_1\in N_1e_1\;\!$, $T\circ\r_1(x_1)\circ V^*\in \cb (H_2)$ is
in the commutant of $\r_2( \pi_2 (Z) )\;\!$, which is by the choice of $\r_2$ equal to $\r_2(N_2)\;\!$.
Moreover, by (\ref{def.V}) we have $T\circ\r_1(x_1)\circ V^*\circ \r_2(e_2) = T\circ\r_1(x_1)\circ V^*$,
so actually $T\circ\r_1(x_1)\circ V^*\in \r_2(N_2e_2)\;\!$.
Therefore the equation
\begin{equation}
\label{ro021}
\r_2\big( \mathcal{L}^{N_1e_1,N_2e_2}(T)\;\! x_1\big):=T\circ\r_1(x_1)\circ V^*\;\! ,\quad x_1\in N_1e_1
\end{equation} 
defines a map $\mathcal{L}^{N_1e_1,N_2e_2}(T) : N_1e_1\longrightarrow N_2e_2$. Straightforward verification shows that
$\mathcal{L}^{N_1e_1,N_2e_2}(T)$ belongs to $\cb_Z\big(N_1e_1,N_2e_2\big)$ and its norm is $\leq \| T\|\;\!$.

Similarly, for each $S\in (N_2 ,\pi_2 ; N_1 ,\pi_1)\;\!$,the equation
\begin{equation}
\label{ro021star}
\r_1\big(\mathcal{L}^{N_2e_2,N_1e_1}(S)\;\! x_2\big):=S\circ\r_2(x_2)\circ V\;\! ,\quad x_2\in N_2e_2
\end{equation}
defines a map $\mathcal{L}^{N_2e_2,N_1e_1}(S)\in \cb_Z\big(N_2e_2,N_1e_1\big)$.

Clearly, the mappings
\medskip

\centerline{$
\begin{array}{l}\mathcal{L}^{N_1e_1,N_2e_2} : (N_1 ,\pi_1 ; N_2 ,\pi_2)\longrightarrow \cb_Z\big(N_1e_1,N_2e_2\big) , \\
\mathcal{L}^{N_2e_2,N_1e_1} : (N_2 ,\pi_2 ; N_1 ,\pi_1)\longrightarrow \cb_Z\big(N_2e_2,N_1e_1\big)\end{array}
$}
\medskip

\noindent are $Z$-linear. They intertwine the $*$-operation:
\begin{equation}
\label{dupou}
\mathcal{L}^{N_1e_1,N_2e_2}(T)^*=\mathcal{L}^{N_2e_2,N_1e_1}(T^*)\,,\quad T\in (N_1 ,\pi_1 ; N_2 ,\pi_2)\;\! .
\end{equation}
For (\ref{dupou}) we have to show that, for every $x_1\in N_1e_1$ and $x_2\in N_2e_2\;\!$,
\begin{equation}
\label{dupou2}
\langle\;\! \mathcal{L}^{N_1e_1,N_2e_2}(T)\;\! x_1 \!\mid\! x_2\;\! \rangle_{\substack{ {} \\ N_2e_2}} =
\langle\;\! x_1 \!\mid\! \mathcal{L}^{N_2e_2,N_1e_1}(T^*)\;\! x_2\;\! \rangle_{\substack{ {} \\ N_1e_1}}
\end{equation}
holds true. But, since
\medskip

\noindent\hspace{9.3 mm}$\r_2\Big( \pi_2\big( \langle\;\! \mathcal{L}^{N_1e_1,N_2e_2}(T)\;\! x_1 \!\mid\! x_2\;\!
\rangle_{\substack{ {} \\ N_2e_2}} \big)\! \Big) V$
\smallskip

\noindent\hspace{0.5 mm}$\overset{(\ref{def.V})}{=}\r_2\Big( \pi_2\big( \langle\;\!
\mathcal{L}^{N_1e_1,N_2e_2}(T)\;\! x_1 \!\mid\! x_2\;\! \rangle_{\substack{ {} \\ N_2e_2}}\big)\;\! e_2\Big) V$
\smallskip

\noindent\hspace{2.6 mm}$=\;\;\r_2\Big(\! \big( \mathcal{L}^{N_1e_1,N_2e_2}(T)\;\! x_1 \!\mid\! x_2\;\!
\big)_{\!\substack{ {} \\ N_2e_2}} e_2\Big) V
=\r_2(x_2^{\;\! *}) \r_2\big( \mathcal{L}^{N_1e_1,N_2e_2}(T)\;\! x_1\big) V$
\smallskip

\noindent\hspace{0.5 mm}$\overset{(\ref{ro021})}{=}\r_2(x_2^{\;\! *})\circ T\circ\r_1(x_1)\circ (V^*V)
\overset{(\ref{def.V})}{=}(V V^*)\circ \r_2(x_2^{\;\! *})\circ T\circ\r_1(x_1)$
\medskip

\noindent\hspace{2.6 mm}$=\;\; V\circ\big( T^*\! \circ \r_2(x_2)\circ V\big)^*\! \circ \r_1(x_1)$
\smallskip

\noindent\hspace{0.5 mm}$\overset{(\ref{ro021star})}{=}V\circ \r_1\big( \mathcal{L}^{N_2e_2,N_1e_1}(T^*)\;\!
x_2\big)^*\! \circ \r_1(x_1)$
\smallskip

\noindent\hspace{2.6 mm}$=\;\; V \r_1 \Big( e_1 \big( x_1 \!\mid\! \mathcal{L}^{N_2e_2,N_1e_1}(T^*)\;\! x_2
\big)_{\!\substack{ {} \\ N_1e_1}}\Big)$
\smallskip

\noindent\hspace{2.6 mm}$=\;\; V \r_1\Big( e_1 \pi_1\big( \langle\;\! x_1 \!\mid\!
L^{N_2e_2,N_1} e_1)(T^*)\;\! x_2\;\! \rangle_{\!\substack{ {} \\ N_1e_1}}\big)\! \Big)$
\smallskip

\noindent\hspace{0.5 mm}$\overset{(\ref{def.V})}{=}V \r_1\Big( \pi_1\big( \langle\;\! x_1 \!\mid\!
L^{N_2e_2,N_1} e_1)(T^*)\;\! x_2\;\! \rangle_{\!\substack{ {} \\ N_1e_1}}\big)\! \Big)$
\smallskip

\noindent\hspace{0.5 mm}$\overset{(\ref{V-intertw.})}{=}\r_2\Big( \pi_2\big( \langle\;\! x_1 \!\mid\!
L^{N_2e_2,N_1} e_1)(T^*)\;\! x_2\;\! \rangle_{\!\substack{ {} \\ N_1e_1}}\big)\! \Big) V ,$
\smallskip

\noindent using (\ref{def.V}), the injectivity of $\r_2$ and $\pi_2\;\!$, and
$z_{N_2}(e_2)=\idd_{N_2}\;\!$, the proof of (\ref{dupou2}) is straightforward.
\smallskip

We now prove the analogue of Theorem \ref{isotr} for $\cb_Z (N_1e_1,N_2e_2)$.

\begin{thm}
\label{atdu}
Let $Z$ be an abelian $W^*$-algebra and, for $j=1,2$, $N_j$ a type $\ty{I}$ $W^*$-algebra,
$e_j\in N_j$ an abelian projection with central support $z_{N_j}(e_j)=\idd_{N_j}\;\!$,  
$\pi_j:Z\to Z(N_j)$ a $*$-isomorphism, and $(\r_j,H_j)$ a spatial representation of $N_j$
such that  $\r_j(N_j)'=\r_j(Z(N_j))=\r_j(\pi_j(Z))$.

Let further $(N_j ,\pi_j ; N_k ,\pi_k)$, $j,k=1,2$, be defined in \eqref{cucus}, and fix some
$V\in(N_1,\pi_1 ; N_2 ,\pi_2)$ satisfying \eqref{def.V} and \eqref{V-intertw.}.

Then 
\begin{equation}
\label{mod.maps}
\cb_Z\big(N_1e_1,N_2e_2\big) =\big\{\mathcal{L}^{N_1e_1,N_2e_2}(T)\,;\, T\in (N_1 ,\pi_1 ; N_2 ,\pi_2)\big\} ,
\end{equation}
where
$\mathcal{L}^{N_1e_1,N_2e_2}(T) : N_1e_1\longrightarrow N_2e_2$ is defined by $(\ref{ro021})$, and
\begin{equation}
\label{ro21}
\big\| \mathcal{L}^{N_1e_1,N_2e_2}(T)\big\|=\|T\|\,,\quad T\in (N_1 ,\pi_1 ; N_2 ,\pi_2)\;\! .
\end{equation}
Thus $\mathcal{L}^{N_1e_1,N_2e_2} : (N_1 ,\pi_1 ; N_2 ,\pi_2)\longrightarrow \cb_Z\big(N_1e_1,N_2e_2\big)$
is an isometric $Z$-linear isomorphism.
Moreover,
\begin{equation}
\mathcal{L}^{N_1e_1,N_2e_2}(T)^*=\mathcal{L}^{N_2e_2,N_1e_1}(T^*)\,,\quad T\in (N_1 ,\pi_1 ; N_2 ,\pi_2)\;\! ,
\tag{\ref{dupou}}
\end{equation}
where $\mathcal{L}^{N_2e_2,N_1e_1}(T^*)$ is defined by $(\ref{ro021star})$.
\end{thm}

\begin{proof}
Inclusion $\supset$ in (\ref{mod.maps}) was already commented after (\ref{ro021}).

Conversely, fix now $S\in\cb_Z\big(N_1e_1,N_2e_2\big)$. With $N$ and $e$ as in
Proposition \ref{wuatt}, taking into account Proposition \ref{wuatt}, (iii), we can define
$\widetilde{S} : Ne\longrightarrow Ne$ by setting, for $x_1\in N_1e_1\,,x_2\in N_2e_2\;\!$,
$$
\widetilde{S}\begin{pmatrix} 
	 \r_{1}(x_1)& \r_{1}(x_1)V^*\\
	 \r_{2}(x_2)V&  \r_{2}(x_2)\end{pmatrix}:=\begin{pmatrix} 
	0& 0\\
	 \r_{2}(Sx_1)V&  \r_{2}(Sx_1)\end{pmatrix}\;\! .
$$
$\widetilde{S}$ is clearly $Z$-linear and Proposition \ref{wuatt}, (iv), leads to $\|\widetilde{S}\|\leq\|S\|$,
thus $\widetilde{S}$ defines an element of $\cb_Z\big(N_1e_1,N_2e_2\big)$. By Theorem \ref{isotr},
$\widetilde{S}=L^{N e}(R)$ for a unique element $R\in N$.
Since, by Proposition \ref{wuatt}, $R=\begin{pmatrix} 
	R_{11}& R_{12}\\
	R_{21}&  R_{22}\end{pmatrix}$ with $R_{jk}\in(N_k ,\pi_k ; N_j ,\pi_j)$, $j,k=1,2$, we have
$$
\begin{pmatrix} 
	0& 0\\
	 \r_{2}(Sx_1)V&  \r_{2}(Sx_1)\end{pmatrix}=\begin{pmatrix} 
	R_{11}& R_{12}\\
	R_{21}&  R_{22}\end{pmatrix}\begin{pmatrix} \r_1(x_1)& \r_1(x_1) V^*\\
\r_2(x_2) V&\r_2(x_2)\end{pmatrix}
$$
for all $x_1\in N_1e_1$ and $x_2\in N_2e_2$. Comparison of the matrix entries in position $22$ of both
members entails
$$
\r_2(Sx_1)=R_{21}\r_1(x_1)V^*+R_{22}\r_2(x_2)
$$
for any $x_1\in N_1e_1\;\! , x_2\in N_2e_2$. In particular, with $x_2=0$ we have
$$
\r_2(Sx_1)=R_{21}\r_1(x_1)V^*\,,\quad x_1\in N_1e_1\,,
$$
that is $S=\mathcal{L}^{N_1e_1,N_2e_2}(R_{21})$.

To prove \eqref{ro21}, for $x_1\in N_1e_1$ we compute
\begin{align*}
\big\|\mathcal{L}^{N_1e_1,N_2e_2}(T)x_1\big\|^2=&\big\|\r_2\big(\mathcal{L}^{N_1e_1,N_2e_2}(T)x_1\big)\big\|^2
=\|T\circ\r_1(x_1)\circ V^*\|^2\\
\overset{\text{\eqref{def.V}}}{=}&\|T\circ\r_1(x_1)\|^2
=\|\r_1(x_1^*)T^*T\r_1(x_1)\|\,.
\end{align*}
By \eqref{intertw.} and \eqref{bikkk}, $T^*T=\r_1(a_1)$ for some $0\leq a_1\in N_1$, and thus
\begin{align*}
\big\|\mathcal{L}^{N_1e_1,N_2e_2}(T)x_1\big\|^2=&\|\r_1(x_1^*a_1x_1)\|=\|x_1^*a_1x_1\|\\
=&\|a_1^{1/2}x_1\|=\big\|L^{N_1e_1}\big(a_1^{1/2}\big)x_1\big\|^2\,.
\end{align*}
Summarising,
\begin{align*}
\big\|\mathcal{L}^{N_1e_1,N_2e_2}(T)\big\|=
&\|L^{N_1e_1}\big(a_1^{1/2}\big)\big\|\overset{\text{\eqref{op.norm}}}{=}\big\|a_1^{1/2}\big\|\\
=&\|a_1\|^{1/2}=\|T^*T\|^{1/2}=\|T\|\,.
\end{align*}
\end{proof}

With the above notations, let now be $\gx_j=p_{\gx_j}N_je_j\subset N_je_j\;\! ,j=1\;\! ,2\;\! ,$
$s$-closed $Z$-submodules. Similarly as $\mathcal{L}^{N_1e_1,N_2e_2}$ was defined by (\ref{ro021}),

\noindent we obtain a $Z$-linear mapping
\medskip

\centerline{$\mathcal{L}^{\gx_1,\gx_2} : \r_2(p_{\gx_2})\circ(N_1,\pi_1;N_2,\pi_2)\circ\r_1(p_{\gx_1})
\longrightarrow \cb_Z(\gx_1,\gx_2)$}
\medskip

\noindent by defining $\mathcal{L}^{\gx_1,\gx_2}(T)$ for $T\in\r_2(p_{\gx_2})\circ(N_1,\pi_1;N_2,\pi_2)\circ
\r_1(p_{\gx_1})$ via the equation
\begin{equation}
\label{ro022}
\r_2\big(\mathcal{L}^{\gx_1,\gx_2}(T)\;\! x_1\big) :=T\circ\r_1(x_1)\circ V^*\;\! ,\quad x_1\in\gx_1\;\! .
\end{equation}
In the same way, for each $S\in \r_1(p_{\gx_1})\circ(N_2,\pi_2;N_1,\pi_1)\circ \r_2(p_{\gx_2})\;\!$,
the equation
\begin{equation}
\label{ro023}
\r_1\big(\mathcal{L}^{\gx_2,\gx_1}(S)\;\! x_2\big):=S\circ\r_2(x_2)\circ V\,,\quad x_2\in \gx_2
\end{equation}
defines some $\mathcal{L}^{\gx_2,\gx_1}(S)\in \cb_Z(\gx_2,\gx_1)$ and the obtained mapping
\medskip

\centerline{$\mathcal{L}^{\gx_2,\gx_1} : \r_1(p_{\gx_1})\circ(N_2,\pi_2;N_1,\pi_1)\circ\r_2(p_{\gx_2})
\longrightarrow \cb_Z(\gx_2,\gx_1)$}
\medskip

\noindent is $Z$-linear.

The proof of (\ref{dupou}) can be easily adapted to prove that $\mathcal{L}^{\gx_1,\gx_2}$ and
$\mathcal{L}^{\gx_2,\gx_1}$ intertwine the $*$-operation:
\begin{equation}
\label{dupou0}
\begin{split}
&\mathcal{L}^{\gx_1,\gx_2}(T)^*=\mathcal{L}^{\gx_2,\gx_1}(T^*)\,\,\,\text{for}\\
&T\in\r_2(p_{\gx_2})\circ(N_1,\pi_1;N_2,\pi_2)\circ\r_1(p_{\gx_1})\,,
\end{split}
\end{equation}

Theorem \ref{atdu} allows us to investigate the structure of $\cb_Z(\gx_1,\gx_2)\;\!$.

\begin{cor}
\label{mdol}
Using the notations of Theorem \ref{atdu}, for arbitrary $s$-closed $Z$-submodules
$\gx_1=p_{\gx_1}N_1e_1\subset N_1e_1$ and $\gx_2=p_{\gx_2}N_2e_2\subset N_2e_2$, we have
\begin{equation}
\label{mod.maps2}
\begin{split}
&\cb_Z(\gx_1,\gx_2) \\
=\;&\big\{\mathcal{L}^{\gx_1,\gx_2}(T)\,;\,T\in\r_2(p_{\gx_2})\circ(N_1,\pi_1;N_2,\pi_2)\circ\r_1(p_{\gx_1})\big\}\;\! ,
\end{split}
\end{equation}
where $\mathcal{L}^{\gx_1,\gx_2}(T):\gx_1\to\gx_2$ is defined by $(\ref{ro022})$, and
\begin{equation}
\label{ro22}
\big\| \mathcal{L}^{\gx_1,\gx_2}(T)\big\|=
\|T\|\;\! ,\quad T\in\r_2(p_{\gx_2})\circ(N_1,\pi_1;N_2,\pi_2)\circ\r_1(p_{\gx_1})\;\! .
\end{equation}
Consequently $\mathcal{L}^{\gx_1,\gx_2} : \r_2(p_{\gx_2})\circ(N_1,\pi_1;N_2,\pi_2)\circ\r_1(p_{\gx_1})
\longrightarrow \cb_Z(\gx_1,\gx_2)$ is an isometric $Z$-linear isomorphism.
Moreover,
\begin{equation*}
\begin{split}
&\mathcal{L}^{\gx_1,\gx_2}(T)^*=\mathcal{L}^{\gx_2,\gx_1}(T^*)\,\,\, {\rm for}\\
&T\in\r_2(p_{\gx_2})\circ(N_1,\pi_1;N_2,\pi_2)\circ\r_1(p_{\gx_1})\,,
\end{split} \tag{\ref{dupou0}}
\end{equation*}
where $\mathcal{L}^{\gx_2,\gx_1}(T^*)$ is defined by \eqref{ro023}.
\end{cor}

\begin{proof}
Inclusion $\supset$ in (\ref{mod.maps2}) was already explained above.

For the proof of the reverse inclusion, fix an arbitrary $S\in \cb_Z(\gx_1,\gx_2)\;\!$. Then
\begin{equation*}
S_0 :N_1e_1\ni x_1\longmapsto S(p_{\gx_1} x_1)\in \gx_2=p_{\gx_2}N_2e_2\subset N_2e_2\,,
\end{equation*}
is a bounded $Z$-linear map. Thus, by Theorem \ref{atdu}, there exists some

\noindent $T\in (N_1 ,\pi_1 ; N_2 ,\pi_2)$ satisfying $S_0 = \mathcal{L}^{N_1e_1,N_2e_2}(T)\;\!$.

We now claim
\begin{equation}
\label{stima2}
\mathcal{L}^{N_1e_1,N_2e_2}(T)=\mathcal{L}^{N_1e_1,N_2e_2}\big(\r_2(p_{\gx_2})\circ T\circ\r_1(p_{\gx_1})\big)\,.
\end{equation}
In other words, we have to show that, for all $x_1\in N_1e_1$ and $y_2\in N_2e_2\;\!$, it results
$$
\big(\mathcal{L}^{N_1e_1,N_2e_2}(T)x_1\!\mid\!y_2\big)_{N_2e_2}=
\big(\mathcal{L}^{N_1e_1,N_2e_2}\big(\r_2(p_{\gx_2})\circ T\circ\r_1(p_{\gx_1})\big)\;\!
x_1\,\big|\,y_2\big)_{N_2e_2}
$$
or, equivalently,
\begin{align*}
&\r_2\Big(\! \big(\mathcal{L}^{N_1e_1,N_2e_2}(T)\;\! x_1\,\big|\,y_2\big)_{N_2e_2}e_2\Big)\\
=\; &\r_2\Big(\! \big( \mathcal{L}^{N_1e_1,N_2e_2}\big(\r_2(p_{\gx_2})\circ T\circ\r_1(p_{\gx_1})\big)
x_1\,\big|\,y_2\big)_{N_2e_2}e_2\Big)\;\! .
\end{align*}
Indeed, we have
\medskip

\noindent\hspace{12.6 mm}$\r_2\Big(\big(\mathcal{L}^{N_1e_1,N_2e_2}(T)\;\! x_1\,\big|\,y_2\big)_{N_2e_2}e_2\Big) =
\r_2\Big(y_2^{\;\! *}\big(\mathcal{L}^{N_1e_1,N_2e_2}(T)\;\! x_1\big)\! \Big)$
\smallskip

\noindent\hspace{7 mm}$=\;\, \r_2\big(y_2^{\;\! *}(S_ox_1)\big)
= \r_2\big(y_2^{\;\! *}S(p_{\gx_1}x_1)\big)
=\r_2\Big( y_2^{\;\! *} \big( p_{\gx_2} S(p_{\gx_1}(p_{\gx_1}x_1))\big)\! \Big)$
\smallskip

\noindent\hspace{7 mm}$=\;\, \r_2(y_2^{\;\! *}p_{\gx_2})\r_2(S_o(p_{\gx_1}x_1))
=\r_2(y_2^{\;\! *}p_{\gx_2})\r_2\big( \mathcal{L}^{N_1e_1,N_2e_2}(T)(p_{\gx_1}x_1)\big)$
\smallskip

\noindent\hspace{4.8 mm}$\overset{\text{\eqref{ro021}}}{=}\! \r_2(y_2^{\;\! *}p_{\gx_2})\big(T\circ
(\r_1(p_{\gx_1})\r_1(x_1))\circ V^*\big)$
\smallskip

\noindent\hspace{7 mm}$=\;\, \r_2(y_2^{\;\! *})\Big(\! \big(\r_2(p_{\gx_2})\circ
T\circ\r_1(p_{\gx_1})\big)\circ\r_1(x_1)\circ V^*\Big)$
\smallskip

\noindent\hspace{4.8 mm}$\overset{\text{\eqref{ro021}}}{=}\! \r_2\Big(y_2^{\;\! *}\mathcal{L}^{N_1e_1,N_2e_2}
\big(\r_2(p_{\gx_2})\circ T\circ\r_1(p_{\gx_1})\big)x_1\Big)$
\smallskip

\noindent\hspace{7 mm}$=\;\, \r_2 \Big(\! \big(\mathcal{L}^{N_1e_1,N_2e_2}\big(\r_2(p_{\gx_2})\circ
T\circ\r_1(p_{\gx_1})\big) x_1\,\big|\,y_2\big)_{N_2e_2}e_2\Big)\;\! .$
\medskip

According to Theorem \ref{atdu}, $\mathcal{L}^{N_1e_1,N_2e_2}$ is injective, thus \eqref{stima2} yields
$$
T=\r_2(p_{\gx_2})\circ T\circ\r_1(p_{\gx_1}\in\r_2(p_{\gx_2})\circ(N_1,\pi_1;N_2,\pi_2)\circ\r_1(p_{\gx_1})\,.
$$
Therefore,
\begin{align*}
\r_2(Sx_1)=\;&\r_2(S_ox_1)=\r_2\big(\mathcal{L}^{N_1e_1,N_2e_2}(T)x_1\big)
=T\circ\r_1(x_1)\circ V^* \\
=\;&\r_2\big(\mathcal{L}^{\gx_1,\gx_2}(T)x_1\big)
\end{align*}
for every $x_1\in\gx_1$, that is $S=\mathcal{L}^{\gx_1,\gx_2}(T)$.

Concerning \eqref{ro22}, we have just to note that $\|S_o\|=\|S\|$, hence
$$
\|T\|\overset{\text{\eqref{ro21}}}{=}\big\|\mathcal{L}^{N_1e_1,N_2e_2}(T)\big\|=\|S_o\|=\|S\|\,.
$$
\end{proof}

In the case $N_1=N=N_2$ we can reformulate Corollary \ref{mdol} in the realm of $N$,
without using the spatial representations $\r_1$ and $\r_2\;\!$.

\begin{cor}
\label{mdol2}
Let $N$ be a type $\ty{I}$ $W^*$-algebra, $Z :=Z(N)$ its centre, and, for $j=1\;\! ,2\;\!$, $e_j\in N$
abelian projections of central support $z_N(e_j)=\idd_N$, $v\in N$ some partial isometry satisfying
$v^*v=e_1\;\! ,vv^*=e_2\;\!$, and $p_j\in N$ projections.

With $\gx_j:=p_j Ne_j\;\!$, we have
\begin{equation*}
\cb_Z(\gx_1,\gx_2)=
\big\{ L^{\gx_1,\gx_2}(a)\,;\,a\in p_2 Np_1\big\}\;\! ,
\end{equation*}
where $L^{\gx_1,\gx_2}(a) : \gx_1\to \gx_2$ is defined by
\begin{equation*}
L^{\gx_1,\gx_2}(a)\;\! x_1 :=ax_1v^*\;\! ,\quad x_1\in\gx_1\;\! ,
\end{equation*}
and
\begin{equation*}
\big\| L^{\gx_1,\gx_2}(a)\big\|=
\|a\|\;\! ,\quad a\in p_2 Np_1\;\! .
\end{equation*}
Consequently $L^{\gx_1,\gx_2} : p_2 Np_1 \longrightarrow \cb_Z(\gx_1,\gx_2)$ is an isometric
$Z$-linear

\noindent isomorphism.
Moreover,
\begin{equation*}
L^{\gx_1,\gx_2}(a)^*= L^{\gx_2,\gx_1}(a^*)\,,\quad
a\in p_2 Np_1\;\! ,
\end{equation*}
where $L^{\gx_2,\gx_1}(a^*) : \gx_2\to \gx_1$ is defined by
\begin{equation*}
L^{\gx_2,\gx_1}(a^*)\;\! x_2:=a^*x_2\;\! v\;\! ,\qquad x_2\in \gx_2\;\! .
\end{equation*}
\end{cor}

\begin{proof}
Let us consider a spatial representation $(\r\;\! ,H)$
of $N$ such that $\r (N)'=\r (Z)\;\!$. Set, for $j=1\;\! ,2\;\!$, $N_j:=N$, $\pi_j:=\iota$ with
$\iota:Z\to N$ the natural embedding of $Z$ in $N$, and $(\r_j\;\! ,H_j) := (\r\;\! ,H)\;\!$.

For any $j\;\! ,k=1\;\! ,2\;\!$, we have $(N_j,\pi_j;N_k,\pi_k)=\r (Z)'=\r (N)\;\!$. Thus
$V:=\r (v)$ satisfies (\ref{def.V}) and (\ref{V-intertw.}).

Since

\centerline{$\r_2(p_{\gx_2})\circ(N_1,\pi_1;N_2,\pi_2)\circ\r_1(p_{\gx_1}) =\r (p_2 Np_1)\;\! ,$}
\smallskip

\noindent we have for every $a\in p_2 Np_1$ and $x_1\in\gx_1\;\!$,
\smallskip

\centerline{$\r\big(\mathcal{L}^{\gx_1,\gx_2}(\r (a))\;\! x_1\big)\!\! \overset{(\ref{ro022})}{=}\!\!
\r (a)\circ \r(x_1)\circ \r (v)^*\! =\! \r (a x_1 v^*)\! =\! \r \big( L^{\gx_1,\gx_2}(a)\big) x_1\;\! .$}
\smallskip

\noindent Therefore $\mathcal{L}^{\gx_1,\gx_2}(\r (a))=L^{\gx_1,\gx_2}(a)$ for all
$a\in p_2 Np_1\;\!$.

Similarly,
\smallskip

\centerline{$\r_1(p_{\gx_1})\circ(N_2,\pi_2;N_1,\pi_1)\circ\r_2(p_{\gx_2}) =\r (p_1 Np_2)\;\! ,$}
\medskip

\noindent and $\mathcal{L}^{\gx_2,\gx_1}(\r (b))=L^{\gx_2,\gx_1}(b)$ for all $b\in p_1Np_2\;\!$.

Applying now Corollary \ref{mdol}, all desired statements follow immediately.

\smallskip

\end{proof}



\section{Reduction theory by Hilbert module representations}
\label{sec5}

Let now $M$ be an arbitrary $W^*$-algebra. A {\it decomposition of $M$ along a $W^*$-subalgebra
$\idd_M\in Z\subset Z(M)$, in the realm of Hilbert module representations}, means an embedding
of $M$ into a type $\ty{I}$ $W^*$-algebra $N$ such that $Z=Z(N)$, that is a (necessarily unital) faithful,
normal $*$-homomorphism $\iota : M\!\longrightarrow N$ satisfying $\iota(Z)=Z(N)$. Starting with a
spatial representation $\pi$ of $M$ on a Hilbert space $H$, we get such an embedding by setting
$N=\pi (Z)'$, and $\iota=\pi$ considered as a map $M\!\longrightarrow N$.
Actually, any decomposition $\iota : M\!\longrightarrow N$ along $Z\subset Z(M)$ can be obtained in this way:

\begin{prop}\label{embed}
Let $M$ be a $W^*$-algebra, $\idd_M\in Z\subset Z(M)$ a $W^*$-subalgebra of $Z(M)$, $N$ a type
$\ty{I}$ $W^*$-algebra, and $\iota : M\!\longrightarrow\! N$ a faithful, normal $*$-homomorphism
satisfying the condition $\iota(Z)=Z(N)$.
Then there exist a spatial representation $\pi : M\!\longrightarrow \cb(H)\!$ and a
$*$-isomorphism $\a : N\!\longrightarrow \pi (Z)'$ such that $\pi=\a\circ\iota\;\!$.
Moreover, $\pi$ is uniquely determined by $\iota$, up to unitary equivalence.
\end{prop}
\begin{proof}
Let $\pi_0:N\to \cb(H)$ be a faithful, normal $*$-representation such that
$$
\pi_0(N)'=\pi_0(Z(N))=\pi_0(\iota(Z))
$$
(see e.g. \cite{SZ1}, Corollary 6.5). Putting $\pi:=\pi_0\circ\iota$ and $\a=\pi_0\;\!$, considered as a map
$N\longrightarrow \pi_0(N)\;\!$, we are done.

Let now $\rho : M\!\longrightarrow \cb(K)$ be another spatial representation such that, for some
$*$-isomorphism $\beta : N\longrightarrow \rho (Z)'$ we have $\rho=\beta\circ\iota\;\!$
Since the commutants of both $\pi (Z)'$ and $\rho (Z)'$ are type $\ty{I}_1$ (i.e. abelian)
homogeneous von Neumann algebras, $\beta\circ\a^{-1} : \pi (Z)'\longrightarrow\rho (Z)'$
is a spatial $*$-isomorphism (see e.g. \cite{SZ1}, Theorem 8.7).
Consequently $\rho =(\beta\circ\a^{-1})\circ\pi$ is unitarily equivalent with $\pi\;\!$.

\end{proof}

There are two extreme cases. In the case $Z=\bc\idd_M$, a decomposition of $M$ along $Z$ means
nothing but a unital embedding of $M$ into a type $\ty{I}$ factor, that is a spatial representation of $M$.
On the other hand, in the case $Z=Z(M)$, a decomposition of $M$ along $Z$ is customarily considered
a reduction theory for $M$.

If $M\subset \cb(H)$ is a von Neumann algebra and $\idd_H\in Z\subset Z(M)$ a von Neumann subalgebra,
the inclusion $M\subset N:=Z'$ is a decomposition of $M$ along $Z$. Since 
$Z\subset Z(M')$, we obtain simultaneously also a decomposition $M'\subset N$ of $M'$ along $Z$.
Moreover, noticing that $M'=M'\cap N$, if $e$ is an abelian projection of central support 
$z_N(e)=\idd_H$, by Theorem \ref{isotr} we have
\begin{equation*}
L^{Ne}(M)'=L^{Ne}(M')\, .
\end{equation*}

Thus we conclude that the Hilbert module representation $L^{Ne}$ is compatible with the passage to the
commutant for the von Neumann algebras on $H$ containing $Z$.  

Next we show that the map $L^{Ne}$ does not depend essentially on the choice of $e\;\!$. We recall that
two abelian projections with equal central support are equivalent (e.g. \cite{SZ1}, Proposition 4.10).

\begin{prop}
\label{elleem}
Let $N$ be a type $\ty{I}$ $W^*$-algebra with centre $Z$, $e,f\in N$ abelian projections of central support $z_N(e)=z_N(f)=\idd_N$, and $v\in N$ a partial isometry realising the equivalence
$$
e\sim f:\quad v^*v=e\,, vv^*=f\,.
$$
Then the map
$$
U_v: Nf\ni y\mapsto yv\in Ne
$$
is $Z$-unitary with inverse
$$
U_{v^*}:Ne\ni x\mapsto xv^*\in Nf\,.
$$
Moreover, $U_v$ intertwines $L^{Ne}$ with $L^{Nf}$:
$$
U_vL^{Nf}(a)=L^{Ne}(a)U_v\,,\quad a\in N\,.
$$
\end{prop}
\begin{proof}
Clearly,
$$
Ne\ni x\mapsto xv^*\in Nf	\, .
$$
is the inverse of $U_v$.

Now, for every $x\in Ne$ and $y\in Nf$,
$$
z:=(U_vy\!\mid\!x)_{Ne}=\F_e(x^*yv)
$$
is the unique element of $Z$ verifying
\begin{align*}
ze=&x^*yv=v^*vx^*yv
\iff zf=v(ze)v^*\\
=&vv^*vx^*yvv^*
=vx^*y=(xv^*)^*y\,,
\end{align*}
hence
$$
z=\F_f\big((xv^*)^*y\big)=(y\!\mid\!xv^*)_{\substack{ {} \\ Nf}}=\big(y\!\mid\!(U_v)^{-1}x\big)_{\substack{ {} \\ Nf}}\,.
$$
Consequently, $U_v$ is adjointable with adjoint $(U_v)^*=(U_v)^{-1}=U_{v^*}$.

Furthermore, for every $a\in N$ and $y\in Nf$, 
$$
U_vL^{Nf}(a)y=U_v(ay)=ayv=L^{Ne}(a)(yv)=L^{Ne}(a)U_vy\,.
$$
Thus $U_vL^{Nf}(a)=L^{Ne}(a)U_v$.

\end{proof}


\begin{rem}
With the notations of Corollary \ref{mdol2},
we note that 
$$
p_1\sim p_2 \iff \exists \,\,U\in\cb_Z(\gx_1,\gx_2)\,, Z\text{-unitary}\,.
$$

If this is the case, $U\gx_1=\gx_2$ and $\cb_Z(\gx_2)=U\cb_Z(\gx_1)U^*$.
\end{rem}
\begin{proof}
If $p_1=u^*u$, $p_2=uu^*$ for some $u\in p_2Np_1\subset N$, then $L^{\gx_1,\gx_2}(u)
=:U\in \cb_Z(\gx_1,\gx_2)$ is a $Z$-unitary.

Conversely, suppose that there exist a $Z$-unitary in $\cb_Z(\gx_1,\gx_2)$. By Corollary \ref{mdol2},
there exist $u\in p_2Np_1\subset N$ such that $U=L^{\gx_1,\gx_2}(u)$. We also note that $u^*u$ is
in the reduced algebra $p_1Np_1$.

Since $U^*U=\id_{\gx_1}$, for $x_1,y_1\in\gx_1$ we get
\begin{align*}
(x_1\!\mid\!y_1)_{\gx_1}=&\big(U^*Ux_1\!\mid\!y_1\big)_{\gx_1}=\big(Ux_1\!\mid\!Uy_1\big)_{\gx_2}\\
=&\big(L^{\gx_1,\gx_2}(u)x_1\!\mid\!L^{\gx_1,\gx_2}(u)y_1\big)_{\gx_2}\\
=&(ux_1v^*\!\mid\!uy_1v^*)_{\gx_2}\,.
\end{align*}
On the other hand,
\begin{align*}
(ux_1v^*\!\mid uy_1v^*)_{\gx_2}e_2=\, &vy_1^*u^*ux_1v^*
= v\big((ux_1\!\mid\!uy_1)_{\gx_1}e_1\big)v^* \\
=\, &(ux_1\!\mid\!uy_1)_{\gx_1}e_2
\end{align*}
and, consequently, 
\begin{equation}
\label{yiy}
y_1^*x_1=(x_1\!\mid\!y_1)_{\gx_1}e_1=(ux_1\!\mid\!uy_1)_{\gx_1}e_1=y_1^*u^*ux_1
\end{equation}
for each $x_1$ and $y_1$ in $\gx_1$.
Equation \eqref{yiy} yields
$$
\big(L^{\gx_1}(p_1)x_1\!\mid\!y_1\big)_{\gx_1}=(x_1\!\mid\!y_1)_{\gx_1}=
\big(L^{\gx_1}(u^*u)x_1\!\mid\!y_1\big)_{\gx_1}
$$ 
and, by Corollary \ref{coun}, we conclude that $u^*u=p_1$.

The equality $uu^*=p_2$ is proved in a similar way.
 
\end{proof}

Let us briefly explain, why we call a faithful, normal $*$-homomorphism $\iota$ from a
$W^*$-algebra $M$ in a type $\ty{I}$ $W^*$-algebra $N$, mapping a $W^*$-subalgebra
$\idd_M\in Z\subset Z(M)$ onto the centre $Z(N)$ of $N$, decomposition of $M$ (along 
$Z$, in the realm of Hilbert module representations).

Without loss of generality we can consider $M$ a $W^*$-subalgebra of $N$, and $\iota$
the inclusion map $M\subset N$. Thus we are dealing with a $\ty{I}$ $W^*$-algebra $N$
and a $W^*$-subalgebra $M\subset N$ containing the centre $Z$ of $N$.

If $e\in N$ is any abelian projection of central support $z_N(e)=\idd_N$, then, according to
Theorem \ref{isotr}, $L^{Ne} : M\ni a\longmapsto L^{Ne} (a)\in \cb_Z(Ne)$ is a unital, faithful,
normal $*$-homomorphism, mapping $Z$ onto the centre of $\cb_Z(Ne)$
(which consists of the left multiplication operators on $Ne$ by the elements of $Z=Z(N)$).
Furthermore, by Proposition \ref{elleem}, it is uniquely determined up to $Z$-unitary equivalence.

Now, $\cb_Z(N e)$ can be considered a field of bounded linear operators on a field of Hilbert spaces,
labeled by the Gelfand spectrum $\Omega$ of $Z\;\!$, that is the set of all characters of $Z$, equipped
with the weak${}^*$-topology.

Let us begin with the representation of the Hilbert module $Ne$ in terms of a field of Hilbert
spaces labeled by the elements of $\Omega\;\!$.

For each $\omega\in\Omega\;\!$,
\smallskip

\centerline{$(Ne)\times (Ne)\ni (x\;\! ,y)\longmapsto 
\omega\big( (x\!\mid\!x)_{\substack{ {} \\ Ne}}\big)\in \bc$}
\smallskip

\noindent is a pre-inner product. Let $H_{e,\omega}$ denote the Hilbert space
obtained by the Hausdorff completion of $Ne\;\!$. More precisely, setting
\smallskip

\centerline{$(Ne)_{\omega}:=
\{ x\in Ne\;\! ;\;\! \omega\big( (x\!\mid\!x)_{\substack{ {} \\ Ne}}\big) =0\}$}
\smallskip

\noindent and denoting by $x_{\omega}$ the image of $x\in Ne$ by the canonical mapping

\noindent $Ne\longrightarrow Ne / (Ne)_{\omega}\;\!$, the formula
\begin{equation}\label{inner-product}
( x_{\omega} \!\mid\! y_{\omega})_{\omega}:=\omega\big( (x\!\mid\!y)_{\substack{ {} \\ Ne}}
\big) =\omega\big(\Phi_e(y^*x)\big)\;\! ,\quad x\;\! , y\in Ne\;\! ,
\end{equation}
defines an inner product on $Ne / (Ne)_{\omega}$ and $H_{e,\omega}$ is the Hilbert
space obtained by the completion of $Ne / (Ne)_{\omega}$ with respect to the inner
product (\ref{inner-product}).

Let us consider $\widetilde{H}_e:=\{ (\omega\;\! ,\xi )\;\! ; \omega\in\Omega\;\! ,\xi\in H_{e,\omega}\}\;\!$,
the disjoint union of the Hilbert spaces $H_{e,\omega}\;\! ,\omega\in\Omega\;\!$.
It is a bundle with base $\Omega\;\!$, projection map
$P_e : \widetilde{H}_e\ni (\omega\;\! ,\xi )\longmapsto \omega\in\Omega\;\!$, and fibres
$\{H_{e,\om}\;\! ; \om\in\Om\}\;\!$.

A {\it section} of $P_e$ is a map $s : \Omega\longrightarrow \widetilde{H}_e$ such that
$P_e\big(s(\omega )\big) =\omega$ for all $\omega\in \Omega\;\!$. That means that
each $s(\omega )$ is of the form $s(\omega )=(\omega\;\! ,\xi (\omega))$ with $\xi (\omega )
\in H_{e,\omega}\;\!$. It is called {\it bounded} whenever
$\sup\limits_{\omega\in \Omega}\| \xi (\omega )\|<+\infty\;\!$.

The next result shows that we can identify the Hilbert module $Ne$ with the appropriately continuous sections
of this bundle.

\begin{prop}\label{synth.vect}
For $x\in Ne\;\!$, the map
$\Omega\ni\omega\longmapsto (\omega\;\! ,x_{\omega})\in\widetilde{H}_e$ is a bounded
section of $P_e$ such that all functions
$\Omega\ni\omega\longmapsto (x_{\omega} \!\mid\! y_{\omega})_{\omega}\;\! ,y\in Ne\;\!$,
are continuous. Moreover, 
\begin{equation}\label{fibre-norm}
\| x\| =\sup\limits_{\omega\in D} \| x_{\omega}\|\;\! ,\quad x\in Ne\;\! , \,\,D\subset \Omega\text{ any dense subset}.
\end{equation}

Conversely, if $s : \Omega\ni\omega\longmapsto (\omega\;\! ,\xi (\omega))\in\widetilde{H}_e$
is any bounded section of $P_e$ such that all functions
$\Omega\ni\omega\longmapsto (\xi (\omega) \!\mid\! y_{\omega})_{\omega}\;\! ,y\in Ne\;\!$,
are continuous, then there exists a unique $x\in Ne$ such that
$s(\omega )=(\omega\;\! ,x_{\omega})\;\! ,\omega\in\Omega\;\!$.
\end{prop}

\begin{proof}
For $x\in Ne\;\!$, $\Omega\ni\omega\longmapsto (\omega\;\! ,x_{\omega})\in\widetilde{H}_e$
is clearly a section of $P_e\;\!$. Its boundedness is consequence of (\ref{fibre-norm}), which
follows by using (\ref{op-norm}):
\smallskip

\centerline{$\| x\|^2\! \overset{\rm (\ref{op-norm})}{=}\! \big\| (x\!\mid\!x)_{\substack{ {} \\ Ne}}\big\|
=\big\| \Phi_e(x^*x)\big\| =\sup\limits_{\omega\in \Omega} \omega\big(\Phi_e(x^*x)\big)
=\sup\limits_{\omega\in D} \omega\big(\Phi_e(x^*x)\big)$}

\noindent\hspace{1.2 cm}$\overset{\rm (\ref{inner-product})}{=}
\sup\limits_{\omega\in D}\;\! ( x_{\omega} \!\mid\! x_{\omega})_{\omega}
= \sup\limits_{\omega\in D} \| x_{\omega}\|^2 .$
\smallskip

\noindent On the other hand, for every $y\in Ne\;\!$,  the continuity of the function
\smallskip

\centerline{$\Omega\ni\omega\longmapsto (x_{\omega} \!\mid\! y_{\omega})_{\omega}
\overset{\rm (\ref{inner-product})}{=}\omega\big(\Phi_e(y^*x)\big)\in \bc$}
\smallskip

\noindent follows by the definition of the topology of $\Omega\;\!$.

Conversely, let $s : \Omega\ni \omega\longmapsto (\omega\;\! ,\xi (\omega ))\in\widetilde{H}_e$
be a section of $P_e$ such that $c:=\sup\limits_{\omega\in \Omega}\| \xi (\omega )\|<+\infty$ and
all functions
$\Omega\ni\omega\longmapsto (\xi (\omega) \!\mid\! y_{\omega})_{\omega}\;\! ,y\in Ne\;\!$,
are continuous. For each $y\in Ne$, we denote by $F(y)$ the element of $Z$ whose Gelfand
transform is $\Omega\ni \omega\longmapsto (y_{\omega} \!\mid\! \xi (\omega ))_{\omega} =
\overline{(\xi (\omega ) \!\mid\! y_{\omega})_{\omega}}\in \bc\;\!$.
$F : Ne\ni y\longmapsto F(y)\in Z$ is clearly a $Z$-module map.
It is bounded because, for $y\in Ne$, we have $\| F(y)\| =
\sup\limits_{\omega\in \Omega} |\omega\big( F(y)\big) | =
\sup\limits_{\omega\in \Omega} |(y_{\omega} \!\mid\! \xi (\omega ))_{\omega}|
\leq c \| y\|\;\!$.

Now, according to \cite{SZ}, Lemma 1.11, there exists a unique $x\in Ne$ such that
$F(y)=\F_e(x^*y)$ for all $y\in Ne\;\!$. Then
\smallskip

\centerline{$ (y_{\omega} \!\mid\! \xi (\omega ))_{\omega} =\omega\big( F(y)\big) =
\omega\big( \F_e(x^*y)\big) =\omega\big( (y\!\mid\! x)_{\substack{ {} \\ Ne}}\big)
=(y_{\omega} \!\mid\! x_{\omega})_{\omega}$}
\smallskip

\noindent for all $y\in Ne$ and all $\omega\in\Omega\;\!$, hence $\xi (\omega )=x_{\omega}$
for all $\omega\in\Omega\;\!$.

\end{proof}

Now we go to represent the elements of $N$, which can be identified with $\cb_Z(N e)$
according to Theorem \ref{isotr}, as fields of operators acting on the Hilbert spaces
$H_{e,\omega}\;\! ,\omega\in \Omega\;\!$.
\smallskip

For each $\om\in\Omega$ and $a\in N$ we have
\begin{equation*}
\| (ax)_{\omega}\|^2=\omega\big( \F_e(x^*a^*a\;\! x)\big)\leq \| a\|^2\omega\big( \F_e(x^*x)\big)
=\| a\|^2\| x_{\omega}\|^2,\,\, x\in Ne\,,
\end{equation*}
so $Ne / (Ne)_{\omega}\ni x_{\omega}\longmapsto (ax)_{\omega}\in Ne / (Ne)_{\omega}$ extends
to a  linear operator \\ $\pi_{e,\omega}(a) : H_{e,\omega}\longrightarrow H_{e,\omega}\;\!$.
It is easy to check that $\pi_{e,\omega} : N\longrightarrow \cb(H_{e,\omega})$ is
a unital $*$-homomorphiam.

We notice that (\ref{op.norm}) and (\ref{fibre-norm}) imply the operator norm counterpart of the
vector norm formula (\ref{fibre-norm}):
\begin{equation}\label{fibre-op.norm}
\| a\| =\sup\limits_{\omega\in D} \| \pi_{e,\omega}(a)\|\;\! ,\quad a\in N\;\! ,\, D\subset \Omega\text{ any
dense subset}\,.
\end{equation}
Indeed, inequality $\geq$ is consequence of $\|\pi_{e,\omega}\|\leq 1\;\!$. On the other hand, denoting
$c:= \sup\limits_{\omega\in D} \| \pi_{e,\omega}(a)\|\;\!$, we obtain successively

\noindent\hspace{9.1 mm}$\| (ax)_{\omega}\| = \|\pi_{e,\omega}(a)\;\! x_{\omega}\| \leq
c\;\! \| x_{\omega}\|\! \overset{\rm (\ref{fibre-norm})}{\leq}\! c\;\! \| x\|\;\! ,\quad x\in Ne\;\! ,
\omega\in \Omega\;\! ,$
\smallskip

\noindent\hspace{13.1 mm}$\| ax\| \overset{\rm (\ref{fibre-norm})}{=} \sup\limits_{\omega\in D}
\| (ax)_{\omega}\| \leq c\;\! \| x\|\;\! ,\quad x\in Ne$
\smallskip

\noindent and, according to (\ref{op.norm}), $\| a\|\leq c$ follows.
\smallskip

Clearly, for every $z\in Z$ and $\omega\in\Omega$, it holds
\begin{equation}\label{onZ}
(z x)_{\omega}=\omega (z) x_{\omega}\text{ for all }x\in Ne\Longleftrightarrow
\pi_{e,\omega}(z)=\omega (z) \idd_{H_{e,\omega}}\;\! .
\end{equation}

The operator analogous of $\widetilde{H}_e$ is
\smallskip

\centerline{$\widetilde{N}_e :=
\{ (\omega\;\! ,T)\;\! ;\omega\in\Omega\;\! , T\in \cb(H_{e,\omega})\}\;\! ,$}
\smallskip

\noindent the disjoint union of the type $\ty{I}$ factors $\cb(H_{e,\omega})\;\! ,\omega\in\Omega\;\!$.
Similarly to $\widetilde{H}_e$, it is a bundle with base space $\Omega\;\!$, projection map
$\mathcal{P}_{\! e} : \widetilde{N}_e\ni (\omega\;\! ,T )\longmapsto \omega\in\Omega\;\!$,
and fibres $\{ \cb(H_{e,\omega}) ; \omega\in\Omega\}\;\!$.

A {\it section} of $\mathcal{P}_{\! e}$ is a map $\mathcal{S} : \Omega\longrightarrow \widetilde{N}_e$
such that $\mathcal{P}_{\! e}\big( \mathcal{S}(\omega )\big) =\omega$ for all $\omega\in \Omega\;\!$.
This means that each $\mathcal{S}(\omega )$ is of the form $\mathcal{S}(\omega )=(\omega\;\! , T(\omega))$
with $T(\omega )\in \cb(H_{e,\omega})\;\!$. It is called {\it bounded} if
$\sup\limits_{\omega\in \Omega}\| T(\omega )\|<+\infty\;\!$.

The following result is the operator version of Proposition \ref{synth.vect}.

\begin{prop}\label{synth.oper}
For each $a\in N$, the map $\Omega\ni\omega\longmapsto \big(\omega\;\! ,\pi_{e,\omega}(a)\big)
\in \widetilde{N}_e$ is a bounded section of $\mathcal{P}_e$ such that all functions
$\Omega\ni\omega\longmapsto \big( \pi_{e,\omega} (a) x_{\omega} \!\mid\! y_{\omega})_{\omega}\;\!$,
$x\;\! ,y\in Ne\;\!$, are continuous.

Conversely, if $\mathcal{S} : \Omega\ni\omega\longmapsto (\omega\;\! , T(\omega))\in\widetilde{N}_e$
is a bounded section of $\mathcal{P}_{\! e}$ \\ such that all functions
$\Omega\ni\omega\longmapsto \big( T(\omega )\;\! x_{\omega} \!\mid\! y_{\omega})_{\omega}\;\!$,
$x ,y\in Ne$, are continuous, then there exists a unique $a\in N$ such that $\mathcal{S}(\omega )=
\big(\omega\;\! ,\pi_{e,\omega}(a)\big)\;\! ,\omega\in\Omega\;\!$.
\end{prop}


\begin{proof}
For $a\in N$, $\Omega\ni\omega\longmapsto \big(\omega\;\! ,\pi_{e,\omega}(a)\big)\in
\widetilde{N}_e$ is clearly a section of  $\mathcal{P}_e$ and it is bounded by (\ref{fibre-op.norm}).
On the other hand, for every $x ,y\in Ne\;\!$, the continuity of the function
\smallskip

\centerline{$\Omega\ni\omega\longmapsto \big( \pi_{e,\omega} (a) x_{\omega} \!\mid\! y_{\omega})_{\omega}
=\big( (ax)_{\omega} \big|\;\! y_{\omega}\big)_{\! \omega}=\omega\big( \F_e(y^*\! a x)\big)$}
\smallskip

\noindent is consequence of the definition of the topology of $\Omega\;\!$.

Conversely, let $\Omega\ni \omega\longmapsto \mathcal{S}(\omega )=(\omega\;\! ,T(\omega))
\in\widetilde{N}_e$ be a section of $\mathcal{P}_e$ such that
$c:=\sup\limits_{\omega\in \Omega}\| T(\omega )\| <+\infty$ and all functions
$\Omega\ni\omega\longmapsto \big( T(\omega )\;\! x_{\omega} \!\mid\! y_{\omega})_{\omega}\;\!$,
$x ,y\in Ne$, are continuous.

For each $x\in Ne\,,$ $\Omega\ni\omega\longmapsto \big( \omega , T(\omega )x_{\omega}\big)
\in \widetilde{H}_e$ is a bounded continuous section of $P_e\;\!$ with

\centerline{$\| T(\omega ) x_{\omega} \|\leq \| T(\omega )\|\;\! \| x_{\omega}\|\!
\overset{\rm (\ref{fibre-norm})}{\leq}\! c\;\! \| x\|\;\! ,\quad \omega\in \Omega\;\! .$}
\medskip

\noindent Applying Proposition \ref{synth.vect}, we infer that there exists a unique $F(x)\in Ne$ such that
$T(\omega ) x_{\omega} = F(x)_{\omega}\;\! ,\omega\in \Omega\;\! .$ Since
\smallskip

\centerline{$\| F(x)\|\! \overset{\rm (\ref{fibre-norm})}{=}\! \sup\limits_{\omega\in \Omega} \| F(x)_{\omega}\|
=\sup\limits_{\omega\in \Omega} \| T(\omega ) x_{\omega}\|\leq c\;\! \| x\|\;\! ,$}
\smallskip

\noindent $Ne\ni x\longmapsto F(x)\in Ne$ is a bounded $Z$-linear mapping.
Now, according to Theorem \ref{isotr}, there exists a unique $a\in N$ such that
$F(x)=L^{Ne}(a) x=a x$ for all $x\in Ne\;\!$. 
We conclude that, for all $x\in Ne$ and $\omega\in \Omega\;\!$,
\begin{equation*}
T(\omega ) x_{\omega}=F(x)_{\omega} =(ax)_{\omega}=\pi_{e,\omega}(a) x_{\omega}\;\! .
\end{equation*}
Consequently, $T(\omega )=\pi_{e,\omega}(a)$ for all $\omega\in \Omega\;\!$.

\end{proof}

Summing up the aboves, a {\it decomposition of a $W^*$-algebra $M$ along a $W^*$-subalgebra
$\idd_M\in Z\subset Z(M)$, in the realm of Hilbert module representations}, is an identification of $M$ with
a $W^*$-subalgebra of a type $\ty{I}$ $W^*$-algebra $N$ of centre $Z\;\!$, or, equivalently,
an identification of $M$ with a $W^*$-subalgebra of $\cb_Z(Ne)$ for some (does not matter which)
abelian projetion $e\in N$ of central support $z_N(e)\! =\!\idd_N\;\!$, such that $Z$ is identified
with the centre of $\cb_Z(Ne)\;\!$.
According to Theorem \ref{isotr} and Proposition \ref{synth.oper}, $\cb_Z(Ne)$ can be represented
in terms of the bundle $\widetilde{N}_e\;\!$, as the algebra of all appropriately continuous sections
of $\mathcal{P}_e\;\!$. Thus $M$  is identified with an algebra of continuous sections of $\mathcal{P}_e$
such that $Z$ corresponds to the scalar valued sections (cf. (\ref{onZ})).
This leads to a true decomposition of $M$.

\begin{prop}\label{decomp}
Let $M\subset N$ be a $W^*$-subalgebra containing $Z=Z(N)\;\!$.
For each $\omega\in\Omega\;\!$, let us consider any von Neumann algebra $M_{\omega}\subset
\cb(H_{e,\omega})$ containing $\pi_{e,\omega} (M)''$ and contained in $\pi_{e,\omega} (M'\cap N)'$.

Then, for each $a\in M$, the map $\Omega\ni\omega\longmapsto \big(\omega\;\! ,\pi_{e,\omega}(a)\big)
\in \widetilde{N}_e$ is a bounded section of $\mathcal{P}_e$ such that all functions
$\Omega\ni\omega\longmapsto \big( \pi_{e,\omega} (a)\;\! x_{\omega} \!\mid\! y_{\omega})_{\omega}\;\!$,
$x\;\! ,y\in Ne\;\!$, are continuous.

Conversely, if $\mathcal{S} : \Omega\ni\omega\longmapsto (\omega\;\! , T(\omega))\in\widetilde{N}_e$
is a bounded section of $\mathcal{P}_{\! e}$ \\ such that $T(\omega )\in M_{\omega}\;\! ,\omega\in\Omega\;\!$,
and all functions
$\Omega\ni\omega\longmapsto \big( T(\omega )\;\! x_{\omega} \!\mid\! y_{\omega})_{\omega}\;\!$,
$x ,y\in Ne$, are continuous, then there exists a unique $a\in M$ such that $\mathcal{S}(\omega )=
\big(\omega\;\! ,\pi_{e,\omega}(a)\big)\;\! ,\omega\in\Omega\;\!$.
\end{prop}

\begin{proof}
We have just to apply Proposition \ref{synth.oper}, observing that, for $a\in N$, if $\pi_{e,\omega}(a)
\in M_{\omega}\;\! ,\omega\in\Omega\;\!$, then $a\in M$. But this follows by noticing that
$\pi_{e,\omega}(a)\in M_{\omega}\;\! ,\omega\in\Omega\;\!$, and (\ref{fibre-op.norm}) imply
$a\in (M'\cap N)'\cap N=M$ (for the last equality see e.g. \cite{D1}, Part III, Ch. 7, Ex. 13 b).

\end{proof}

Since we do not know whether $\pi_{e,\omega} (M)''=\pi_{e,\omega} (M'\cap N)'$ and, for
$\omega$ not isolated point of $\Omega\;\!$, $H_{e,\omega}$ can have huge dimension
(cf. \cite{FF}, Theorem 1; \cite{T1}; \cite{T}, Theorem V.5.1),
it seems us preferable to work only with those elements of the algebras $\cb(H_{e,\omega})$
which can be "continuously connected with elements of the neighbour algebras", that is with
continuous sections of $\mathcal{P}_{\! e}\;\!$. This means to work with
the Hilbert module $Ne$ and module maps on it.


\section{The reduction of the standard representation}
\label{sec6}

We shall call an {\it antiunitary operator} on a complex Hilbert space $H$ any isometrical,
bijective, antilinear operator $J:H\longrightarrow H\,$. If $J$ is also involutive,
that is $J^{-1}=J\,$, then it is a conjugation.

A von Neumann algebra $M\subset\cb(H)$ is said to be in {\it standard
form} if there exists a conjugation $J$ on $H$ such that
\medskip

\centerline{$JMJ=M'\text{ and } JzJ=z^*,\quad z\in Z(M)\, .$}
\smallskip

Abelian von Neumann algebras are in standard form if and only if they are
maximal abelian von Neumann algebras (se e.g, \cite{SZ1}, E.7.16 or \cite{KR2}, Theorem 9.3.1).
For general von Neumann algebras
the following standardness criterion was proved in \cite{FZ}, Theorem 2.4 :

\begin{thm}
\label{st}
Let $M\subset\cb(H) $ be a von Neumann algebra. If there exists a bijective antilinear
operator $T:H\longrightarrow H$ such that
\begin{equation}
\label{thm-cond}
TMT^{-1}=M'\text{ and }\, TzT^{-1}=z^*\text{ for all }z\in Z(M)\;\! ,
\end{equation}
then $M$ is in standard form.
\end{thm}

\noindent\hspace{12.33 cm}$\square$

\begin{rem}
We notice that for the standardness of a abelian von Neumann algebra or of
a factor $M\subset\cb(H)$ it is enough to assume the existence of a bijective antilinear
operator $T: H\longrightarrow H$ satisfying only the condition
\medskip

\centerline{$TMT^{-1}=M'\, .$}
\smallskip

\noindent Nevertheless, in the general case - even for finite-dimensional $M$ - we need
both conditions.
\end{rem}
\begin{proof}
Indeed, if $M$ is abelian and $M'=TMT^{-1}$ for some bijective
map $T:H\longrightarrow H$, then also $M'$ is abelian and therefore $M$ is
a maximal abelian von Neumann algebra.
On the other hand, if $M$ is a factor
then the condition $TzT^{-1}=z^*, z\in Z(M)\,$, is trivially fulfilled for any bijective,
antilinear operator $T$ on $H\,$.

Now we shall give an example of a von Neumann algebra $M\subset\cb(H)$
which is not in standard form, but for which there exists a conjugation $J$ on $H$
satisfying the condition
\medskip

\centerline{$JMJ=M'\, .$}
\smallskip

Let $N\subset\cb(K)$ be a von Neumann algebra not in standard form, and $J_0$ a conjugation on
$K$ such that $J_0NJ_0=N$ (we can choose, for example, the finite-dimensional
Hilbert space $K:=\bc^2$, $N:=\cb(K)$ or $N:=\bc\;\! \idd_{K}\,$, and any
conjugation $J_0$ on $\bc^2$).
Set $H :=K\bigoplus K$ and
\medskip

\centerline{$M:=\big\{ x\oplus x'\, ;\, x\in N\;\! ,x'\in N'\big\}\subset\cb(H)\, .$}
\smallskip

Considering the conjugation
\medskip

\centerline{$J : H\ni \xi\oplus\eta \longmapsto (J_0\eta )\oplus (J_0\xi )\in H\, ,$}
\smallskip

\noindent we have
\medskip

\centerline{$J(x\oplus x')J =(J_0x'J_0)\oplus (J_0xJ_0)\, ,\quad x\in N\;\! ,x'\in N'\,,$}
\medskip

\noindent and thus the condition $JMJ=M'$ is satisfied. However $M$ is not in standard
form, because, denoting by $p$ the central projection $\idd_{K}\oplus 0_{K}$
of $M$, the reduced/induced von Neumann algebra $M_p$ is spatially

\noindent $*$-isomorphic to $N$ and therefore not in standard form.

\end{proof}
\begin{rem}
In Proposition 2.3 of \cite{FZ} it was proved that a properly infinite von Neumann algebra
$M\subset \cb(H)$ is in standard form whenever there exists a multiplicative antilinear
isomorphism of $M$ onto $M'$, which acts on the centre of $M$ as the $*$-operation.
Here the condition on $M$ of being properly infinite cannot be dropped.
\end{rem}
\begin{proof}
Indeed, let $R\subset \cb(H)$ be the hyperfinite type $\ty{II_1}$ factor in standard form, and consider
the von Neumann algebra $R\otimes \bc\subset \cb(H\otimes\bc^2)$. Obviously, there exists a
multiplicative antilinear $*$-isomorphism of $R\otimes \bc\sim R$ onto $(R\otimes\bc)'=R'\otimes\bm_2(\bc)\;\!$.

Suppose that  $R\otimes \bc$ acts in standard form on $H\otimes\bc^2$.
Let $\xi_0:=\xi\oplus\eta$ be a trace vector in $H\otimes\bc^2$, with $J_{\xi_0}$
the corresponding conjugation. 

First, 

\centerline{$\{x\xi\oplus x\eta\mid x\in R\}\subset H\otimes\bc^2$}
\medskip
\noindent
 is dense in $H\otimes\bc^2$.
\medskip
\noindent
Second,

\centerline{$J_{\xi_0}(x\xi\oplus x\eta)=x^*\xi\oplus x^*\eta\, ,\quad x\in R.$}
\medskip

Consider the selfadjoint projections $p,q\in R'\otimes\bm_2(\bc)$ given by 
\begin{equation*}
p=
\left(
\begin{array}{ll}
\idd_H & 0 \\
0 & 0
\end{array}
\right),\quad
q=
\left(
\begin{array}{ll}
0 & 0 \\
0 & \idd_H
\end{array}
\right).
\end{equation*}
We have $pJ_{\xi_0}q=0$. In fact, on a dense subspace we get
$$
pJ_{\xi_0}q(x\xi\oplus x\eta)=pJ_{\xi_0}(0\oplus x\eta)=p(0\oplus x^*\eta)=0.
$$
Thus
$$
J_{\xi_0}=pJ_{\xi_0}p\oplus qJ_{\xi_0}q.
$$

Therefore, on one hand 
$$
x':=\left(
\begin{array}{ll}
0 & \idd_H\\
0 & 0
\end{array}
\right)\in (R\otimes \bc)',
$$
and on the other hand $x'\notin J_{\xi_0}(R\otimes \bc)J_{\xi_0}$, which is a contradiction.

\end{proof}

Let now $M\subset \cb(H)$ be a von Neumann algebra, and $\idd_H\in Z\subset Z(M)$ a
von Neumann subalgebra.

\begin{defin}
\label{dekz}
We say that $M$ is in  {\it standard form reduced along $Z$} if there exist an abelian projection $e$
of $N:=Z'$ with $z_N(e)=\idd_H\;\!$, and an involutive $Z$-antiunitary map $T:Ne\to Ne\;\!$, such that
$$
Te=e
$$
and
\begin{equation}
\label{czaza0}
\begin{split}
&TL^{Ne}(M)T=L^{Ne}(M')\, ,\\
&TL^{Ne}(z)T=L^{Ne}(z^{*})\,,\quad z\in Z(M)\, .
\end{split}
\end{equation}
\end{defin}

We notice that any standard von Neumann algebra $M\subset \cb(H)$ is in standard form
reduced along $\bc\;\!\idd_H$.
For let us assume that $H\neq\{0\}$ and $J$ is a conjugation on $H$ such that
\medskip

\centerline{$JMJ=M'\text{ and }JzJ=z^*,\quad z\in Z(M)\, .$}
\smallskip

\noindent Then there exists some $\xi\in H$ such that
$$
\xi_0:=\frac12(\xi+J\xi)\neq0\, ,
$$
because otherwise $J=-\idd_H$ would be simultaneously antilinear and linear, contradicting $H\neq\{0\}$.
Then the orthogonal projection $e$ onto $\bc\xi_0$ is an abelian projection of 
$N:=(\bc\xi_0)'=\cb(H)$ with $z_N(e)=\idd_H$. Defining $T:Ne\to Ne$ by
$$
Tx:=JxJ\,,\quad x\in N_e\,,
$$
it is easily seen that $Te=e$ and \eqref{czaza0} is verified.
\smallskip

The goal of this section is to extend Theorem \ref{st} to standard forms reduced along
arbitrary von Neumann subalgebras of the centre. In particular, we will see that if $M$
is a von Neumann algebra and $Z$ is any von Neumann subalgebra of its centre
$Z(M)\;\!$, then $M$ is in standard form if and only if it is in standard form reduced
along $Z\;\!$.

\begin{prop}
\label{prrci}
Let $Z\subset \cb(H)$ be an abelian von Neumann algebra, and $\xi_0\in H$ cyclic unit vector
for $Z'$. Then there exists a unique abelian projection $e$ of $Z'$ satisfying $e\;\!\xi_0=\xi_0\;\!$
and, necessarily, $z_{Z'}(e)=\idd_{H}\;\!$.

If $M\subset \cb(H)$ is a von Neumann algebra with $Z\subset Z(M)$, and $\xi_0$ is a
cyclic and separating vector for $M$, then for the unique abelian projection $e$ of $Z'$
satisfying $e\;\!\xi_0=\xi_0\;\!$, also the equality $J_{\xi_0} eJ_{\xi_0}=e\,$ holds true, where
$J_{\xi_0}$ is the modular conjugation of $M$ associated to $\xi_0$.
\end{prop}
\begin{proof}
Put $N:=Z'$ and consider on it the vector state $\om:=\om_{\xi_0}$, together with the linear
forms on $Z$ given by $\om(\,{\bf\cdot}\,x)$, $x\in N$. Using the Schwarz
inequality repeatedly, we get $0\leq \om(\,{\bf\cdot}\,x)\leq \| x\|\;\!\om$ for every
$0\leq x\in N\,$:
\medskip

\noindent\hspace{11.7 mm}$\om(zx)=\om\big( z^{1/2}(z^{1/2}x)\big)$

\noindent\hspace{6.9 mm}$\leq \om (z)^{1/2}\om(zx^2)^{1/2}
=\om (z)^{1/2}\om\big( z^{1/2}(z^{1/2}x^2)\big)^{1/2}$
\smallskip

\noindent\hspace{6.9 mm}$\leq\om (z)^{1/2 +1/2^2}\om(zx^{2^2})^{1/2^2}$
\smallskip

\noindent\hspace{12.5 mm}$\hdots$
\smallskip

\noindent\hspace{6.9 mm}$\leq\om (z)^{1/2 +1/2^2+\, ...\, 1/2^n}\om(zx^{2^n})^{1/2^n}$
\smallskip

\noindent\hspace{6.9 mm}$\leq\om (z)^{1/2 +1/2^2+\, ...\, 1/2^n}\| z\|^{1/2^n} \| x\| \stackrel{n}{\longrightarrow}
\om (z) \| x\|\, ,\quad 0\leq z\in Z\, .$
\medskip

\noindent Then by the Radon-Nikodym Theorem, for every $0\leq x\in N$, there
exists a unique element $0\leq E(x)\in Z$ such that $\om(\,{\bf\cdot}\,x)=
\om(\,{\bf\cdot}\,E(x))$. It is straightforwardly seen that the map
$x\mapsto E(x)$, recovered as above, extends by linearity to a normal
conditional expectation $E : N\to Z\,$.

Consider the projection $e:=[Z\;\!\xi_0]\in N$. For every $x\in N\,$, we get
\begin{align*}
( exez_1\xi_0\;\! |\;\! z_2\xi_0 ) =\,&
( xz_1\xi_0\;\! |\;\! z_2\xi_0 ) =\om (z_2^*z_1^{} x) \\
=\,&\om \big( z_2^*z_1^{} E(x)\big) =( E(x)z_1\xi_0\;\! |\;\! z_2\xi_0 ) \\
=\,&( E(x)ez_1\xi_0\;\! |\;\! z_2\xi_0 )\, ,\quad z_1\;\! ,z_2\in Z\, ,
\end{align*}
which implies $exe=E(x)\;\! e\in Ze\,$. Therefore, $e$ is an abelian
projection in $N$. We have
$$
z_N(e)=[NZ\xi_0]=[N\xi_0]=\idd_H
$$
and, plainly, $e\;\!\xi_0 =\xi_0\,$.
\smallskip

Let now $f\in N$ be any other abelian projection with $f\xi_0=\xi_0$. Then
$$
z\xi_0=zf\xi_0=fz\xi_0\,,\quad z\in Z\, ,
$$
so $f\geq e\;\!$. Being $e\;\! ,f$ abelian projections with
$z_N(e) =\idd_H\,$, it follows that $e=f\;\! z_N(e) =f\,$.
\smallskip

Suppose now that $\xi_0$ is cyclic and separating for $M$ with modular conjugation $J_{\xi_0}$.
As $J_{\xi_0} zJ_{\xi_0}=z^*,z\in Z\,$, it follows that $N\ni x\mapsto J_{\xi_0} x^*J_{\xi_0}\in N$
is a $*$-anti-automorphism of $N$. Thus $J_{\xi_0} eJ_{\xi_0}$ is still an 
abelian projection in $N$ with $z_N(J_{\xi_0} eJ_{\xi_0})=\idd_H$ leaving $\xi_0$ fixed.
By the previous part of the proof we conclude that $J_{\xi_0} eJ_{\xi_0}=e\,$.

\end{proof}

\begin{cor}
\label{prrci7}
Let $M$ be a von Neumann algebra acting in standard form on the Hilbert space $H$ with
conjugation $J$. For each von Neumann subalgebra $\idd_M\in Z\subset Z(M)$, there exists
an abelian projection $e$ of $Z'=:N$ with $z_N(e)=\idd$ such that $JeJ=e\;\!$.
\end{cor}

\begin{proof}
By \cite{D1}, Part III, Ch. 1, Lemma 7, there exists a family $(B_\iota,N_\iota)_{\iota\in I}$ of pairs
consisting of a $W^*$-algebra $B_\iota$ of countable type and a type $\ty{I}$ factor $N_\iota$ such that
$$
M\text{ and }\bigoplus_{\iota\in I}\big(B_\iota\overline{\otimes}N_\iota\big)\text{ are }*\text{-isomorphic.}
$$
Choose a faithful normal state $\f_\iota$ on every $B_\iota$ and let $\t_\iota$ denote the \\
canonical trace $\t_\iota$ on $N_\iota\;\!$.
Then $\f:=\bigoplus\limits_{\iota\in I}\big(\f_\iota\otimes\t_\iota\big)$ is a faithful, semifinite,
normal weight on $M$, whose GNS representation is
\begin{equation*}
\pi_\f=\bigoplus\limits_{\iota\in I}\big(\pi_{\f_\iota}\otimes\pi_{\t_\iota}\big) : M\longrightarrow
\cb(H_{\f})\;\! ,
\end{equation*}
where $H_\f=\bigoplus\limits_{\iota\in I}\big(H_{\f_\iota}\otimes H_{\t_\iota}\big)$.
Thus $\bigoplus\limits_{\iota\in I}\big(\pi_{\f_\iota}(B_\iota)\overline{\otimes}\pi_{\t_\iota}(N_\iota)\big)
\subset B\big(H_\f)$ is a standard von Neumann algebra, which is $*$-isomorphic to
$\bigoplus\limits_{\iota\in I}\big(B_\iota\overline{\otimes}N_\iota\big)$, and hence to $M$.
By 10.26 of \cite{SZ1}, there exists a unitary $u:H\to H_\f$ such that
\begin{align*}
uMu^*&=\bigoplus_{\iota\in I}\big(\pi_{\f_\iota}(B_\iota)\;\!\overline{\otimes}\;\!\pi_{\t_\iota}(N_\iota)\big)\,,\\
uJu^*&=\bigoplus_{\iota\in I}\big(J_{\f_\iota}\otimes J_{\t_\iota}\big)=J_\f\,.
\end{align*}
In addition, $uZu^*$ is a subalgebra of 
$$
uZ(M)u^*=\bigoplus_{\iota\in I}\big(\pi_{\f_\iota}(Z(B_\iota))\otimes\bc\idd_{H_{\t_\iota}}\big) ,
$$
so for appropriate $W^*$-subalgebras $\idd_{B_\iota}\in Z_\iota\subset Z(B_\iota)$,
$$
uZu^*=\bigoplus_{\iota\in I}\big(\pi_{\f_\iota}(Z_\iota)\otimes\bc\idd_{H_{\t_\iota}}\big)\,.
$$
According to Proposition \ref{prrci}, for each $\iota\in I$ there is an abelian projection
$e_\iota\in\pi_{\f_\iota}(Z_\iota)'$ satisfying $z_{\pi_{\f_\iota}(Z_\iota)'}(e_\iota)=\idd_{H_{\f_\iota}}$, 
$J_{\f_\iota}e_\iota J_{\f_\iota}=e_\iota$.
\smallskip

On the other hand, for each $\iota\in I$ we have also an abelian projection $0\neq f_\iota\in\pi_{\t_\iota}(\bc\idd_{H_{\t_\iota}})'=\cb(H_{\t_\iota})$, hence 
$z_{\pi_{\t_\iota}(\bc\idd_{H_{\t_\iota}})'}(f_\iota)=\idd_{H_{\t_\iota}}$, such that
$J_{\t_\iota}f_\iota J_{\t_\iota}=f_\iota$. Indeed, it is enough to choose some 
$$
0\neq y_\iota=y_\iota^*\in\ga_{\t_\iota}=\{a_\iota\in N_\iota\,;\,\t(a_\iota^* a_\iota)=
\t(a_\iota a_\iota^*)<+\infty\}\;\! ,
$$
and put for $f_\iota$ the orthogonal projection onto $\bc(y_\iota)_{\t_\iota}\subset H_{\t_\iota}$.
Then $\bigoplus\limits_{\iota\in I}\big(e_\iota\otimes f_\iota\big)$ will be an abelian projection in
$$
\bigg(\bigoplus_{\iota\in I}\big(\pi_{\f_\iota}(Z_\iota)\otimes\bc\idd_{H_{\t_\iota}}\big)\bigg)'
=\bigoplus_{\iota\in I}\big(\pi_{\f_\iota}(Z_\iota)'\overline{\otimes}\cb(H_{\t_\iota})\big)
$$
of central support $\idd_{H_\f}$, satisfying $J_\f\big(\bigoplus\limits_{\iota\in I}\big(e_\iota\otimes f_\iota\big)\big)J_\f
=\bigoplus\limits_{\iota\in I}\big(e_\iota\otimes f_\iota\big)$. Thus
$e:=u^*\big(\bigoplus\limits_{\iota\in I}\big(e_\iota\otimes f_\iota\big)\big)u$ is a projection
we are searching for.

\end{proof}

\begin{prop}
\label{monst}
Let $\idd_H\in Z\subset M, M'\subset Z'=:N\subset \cb(H)$ be inclusions of von Neumann
algebras with $Z$ abelian. Suppose that $M$ $($resp. $M')$ has a cyclic vector
$\xi_0\;\!$, and let $e_0\in N$ be the abelian projection with $e_0\xi_0=\xi_0\;\!$, whose
existence is guaranteed by Proposition \ref{prrci}.

If $T:Ne_0\longrightarrow Ne_0$ is a bijective $Z$-antilinear map satisfying
\begin{equation}
\label{aia}
TL^{Ne_0}(M)T^{-1}=L^{Ne_0}(M')\,,
\end{equation}
then $T^{-1}(e_0)\xi_0$ $($resp. $T(e_0)\xi_0)$ is a separating vector for $M$
$($resp. $M')\;\!$. Therefore, $M$ 
acts in standard form.
\end{prop}
\begin{proof}
We consider only the case of a cyclic vector $\xi_0$ of $M$, the case of a cyclic
vector of $M'$ being reducible to this case, applying it to $M'$ instead of $M$. 

We start by noticing that the map
$$
Ne_0\ni a\longmapsto a\xi_0\in H
$$
is injective. Indeed,  
$$
a\in Ne_0\Rightarrow a^*a\in e_0Ne_0=Ze_0\Rightarrow a^*a=ze_0
$$
for a unique positive element $z\in Z$. Suppose that $a\xi_0=0$. Then
\begin{equation*}
0=ba^*a\xi_0=bze_0\xi_0=bz\xi_0=zb\xi_0\text{ for every }b\in M
\Longrightarrow z=0
\end{equation*}
as $\xi_0$ is cyclic for $M$, so $a^*a=0 \Longleftrightarrow a=0\;\!$.

Let now $x\in M$ be such that $xT^{-1}(e_0)\xi_0=0$ and put 
$$
a:=xT^{-1}(e_0)\in Ne_0\,.
$$ 
Then $a\xi_0=0$ which leads to $a=xT^{-1}(e_0)=0$ for the previous part. Hence 
$T\big(xT^{-1}(e_0)\big)=0$. 

On the other hand, according to \eqref{aia}, $xT^{-1}(e_0)=L^{Ne_0}(x)\big(T^{-1}(e_0)\big)$
implies
\begin{align*}
T\big(xT^{-1}(e_0)\big)=\;&TL^{Ne_0}(x)\big(T^{-1}(e_0)\big) =\big(TL^{Ne_0}(x)T^{-1}\big)e_0 \\
=\;&L^{Ne_0}(x')e_0=x'e_0
\end{align*}
for some $x'\in M'$, uniquely determined by $TL^{Ne_0}(x)T^{-1}=L^{Ne_0}(x')\;\!$.

Since, for every $y\in M$,
$$
0=y\;\! T\big(xT^{-1}(e_0)\big)\xi_0=y\;\! x'e_0\xi_0=y\;\! x'\xi_0=x' y\;\!\xi_0\,,
$$
we have $x'=0$ because $\xi_0$ was supposed to be cyclic for $M$. Thus
$$
L^{Ne_0}(x)=T^{-1}L^{Ne_0}(x')T=0
$$
and Theorem \ref{isotr} leads to $x=0\;\!$.

Now, since $M$ has both a cyclic vector and a separating vector, by the
Dixmier-Mar\'echal Theorem (\cite{DM}, Corollaire 1) we infer that $M$
has a cyclic and separating vector. Hence it acts in standard form.

\end{proof}

Now we go to prove the main result of the present paper, an extension of Theorem \ref{st} to a
characterization theorem for standardness along a von Neumann subalgebra of the centre.

\begin{thm}
\label{st1}
Let $\idd_H\in Z\subset M\subset Z'=:N\subset\cb(H)$ be inclusions of

\noindent von Neumann algebras, with $Z$ abelian. Then the following assertions are
equivalent $:$
\begin{itemize}
\item[(i)] $M$ is in standard form.
\item[(ii)] There exist an abelian projection $e\in N$ of central support

\noindent $z_{N}(e)=\idd_H$, and an involutive, multiplicative, antilinear $*$-map
$j:N\longrightarrow N\!$ fulfilling $j(M)=M'$, $j(z)=z^{*}$ for each $z\in Z(M)$, and $j(e)=e\;\!$.
\item[(iii)] There is a bijective, multiplicative, antilinear map $j:N\longrightarrow N$
fulfilling $j(M)=M'$ and $j(z)=z^{*}$ for each 
$z\in Z(M)$.
\item[(iv)] $M$ is in standard form reduced along $Z$ $($cf. Definition \ref{dekz}$\;\! )$,
that is, for an appropriate abelian projection $e\in N$ with $z_{N}(e)=\idd_H$, there
exists an involutive $Z$-antiunitary map $T:Ne\longrightarrow Ne$ such that $Te=e$ and
\begin{equation}
\begin{split}
&TL^{Ne}(M)T=L^{Ne}(M')\, ,\\
&TL^{Ne}(z)T=L^{Ne}(z^{*})\,,\quad z\in Z(M)\,,
\end{split} \tag{\ref{czaza0}}
\end{equation}
holds true.
\item[(v)] For each abelian projection $e\in N$ with $z_{N}(e)=\idd_H$, there exists an
involutive $Z$-antiunitary map $T:Ne\longrightarrow Ne$ such that \eqref{czaza0} holds true.
\item[(vi)] For each abelian projection $e\in N$ with $z_{N}(e)=\idd_H$, there are an
abelian projection $f\in N$ with $z_{N}(f)=\idd_H$ and
a $Z$-antiunitary map $T:Ne\longrightarrow Nf$ such that $Te=f$ and
\begin{equation}
\label{czaza2}
\begin{split}
&TL^{Ne}(M)T^{-1}=L^{Nf}(M')\, ,\\
&TL^{Ne}(z)T^{-1}=L^{Nf}(z^{*})\,,\quad z\in Z(M)\,.
\end{split} 
\end{equation}
\item[(vii)] There exist abelian projections $e\;\! , f\in N , z_{N}(e)=z_{N}(f)=\idd_H\;\!$,
and a bijective $Z$-antilinear map
$T:Ne\longrightarrow Nf$ such that \eqref{czaza2} holds true.
\end{itemize}
\end{thm}

\begin{proof}
(i)$\;\!\Rightarrow$(ii) As $M$ is acting in standard form on $H$ with conjugation $J$, we can 
show that $JxJ\in Z'$ whenever $x\in N$. Indeed, as
$Z\subset Z(M)$, we get for each $z\in Z$,
$$
JxJz=Jxz^{*}J=Jz^{*}xJ=zJxJ\,.
$$
Thus we can define $j:N\rightarrow N$ by setting
$$
j(x):=JxJ\,,\quad x\in N\,.
$$
Clearly, $j$ is an involutive, multiplicative, antilinear $*$-map satisfying $j(M)=M'$ and $j(z)=z^*$
for every $z\in Z(M)$. Additionally, Corollary \ref{prrci7} guarantees the existence of an abelian
projection $e\in N$ of central

\noindent support $z_{N}(e)=\idd_H$ such that $j(e)=JeJ=e\;\!$.
\smallskip

(ii)$\;\!\Rightarrow$(iii) is obvious.
\smallskip

(iii)$\;\!\Rightarrow$(i) We start by noticing that $j$ is a normal map. Indeed, if $\f$ is a faithful,
semifinite, normal
weight on $N$, $\pi_\f$ the associated GNS representation, and $J_\f$ the corresponding
modular conjugation, then 
$$
\th_\f:\pi_\f(N)\ni\pi_\f(x)\longmapsto J_\f\pi_\f(x)J_\f\in\pi_\f(N)'
$$
is a bijective, multiplicative, antilinear map. Thus
$$
\th_\f\circ\pi_\f\circ j:N\longrightarrow\pi_\f(N)'
$$
is an algebra isomorphism, so it is normal by Theorem II in \cite{O}. Thus 
$$
j=\big(\th_\f\circ\pi_\f\big)^{-1}\circ\big(\th_\f\circ\pi_\f\circ j\big)
$$
is normal too.

To end the proof, we go now to construct a bijective antilinear

\noindent operator $T:H\longrightarrow H$
satisfying condition \eqref{thm-cond} in Theorem \ref{st}.

As $j$ acts on $Z$ as the $*$-operation, the matter can be reduced to the case of
homogeneous $N$. In that case, Section 9 of \cite{K} suggests us how to proceed.

Fix an abelian projection $e\in N$ with $z_{N}(e)=\idd$ and choose a family of partial isometries
$(v_{\iota})_{\iota\in I}$ in $N$ such that
$$
v_{\iota}^{*}v_{\iota}=e\text{ for all }\iota\in I,\quad \sum_{\iota\in I}v_{\iota}v_{\iota}^{*}=\idd\,.
$$
Then $(Z_e)'=N_e=Z_e\;\!$, so $Z_e$ is acting in standard form on $eH$.
Let $J_{e}:eH\longrightarrow eH$ be a modular conjugation of $Z_{e}$.
It fulfils
\begin{equation*}
J_{e}zeJ_{e}=z^{*}e\,,\quad z\in Z\, .
\end{equation*}

Though $j$ is not necessarily a $*$-map, we can adjust it to a $*$-map $j_0 : N\longrightarrow N$
by composition with an inner algebra automorphism of $N$.
For we consider a $*$-representation $\pi : N\longrightarrow\cb (K)$ such that the von Neumann
algebra $\pi (N)$ is in standard form with conjugation $J_K\;\!$.
By Theorem I in \cite{O} there exists an invertible $0\leq a\in N$ such that
$$
\F : N\ni x \longmapsto J_K\pi \big( j(axa^{-1})\big) J_K\in \pi (N)'
$$
is a $*$-isomorphism. Define
$$
j_0(x):=j(axa^{-1})\,,\quad x\in N\;\! .
$$
As $j_0(\,{\bf\cdot}\,)=\pi^{-1}\big(J_K\F(\,{\bf\cdot}\,)J_K\big)$,
it is readily seen that $j_0 : N\longrightarrow N$ is a

\noindent bijective, multiplicative, antilinear $*$-map. It turns out that $a\in Z(M)'$, so
we have also
\begin{equation*}
j_0(z)=j(aza^{-1})\! =j(z)=z^*\;\! ,\quad z\in Z(M)\;\! .
\end{equation*}
Indeed, using that  $j$ is multiplicative, acting on $Z(M)$ as the $*$-operation,
and $j_0$ is a $*$-map, we deduce successively for every $z\in Z(M)\;\!$:
\begin{equation*}
\begin{split}
j(aza^{-1})=\;&j_o(z)=j_0(z^*)^*=j(az^*a^{-1})^*\! =\big( j(a)j(z^*)j(a)^{-1}\big)^* \\
=\;&\big( j(a)zj(a)^{-1}\big)^*=j(a)^{-1}z^*j(a)=j(a)^{-1}j(z)j(a) \\
=\;&j(a^{-1}z\;\! a)\;\! , \\
aza^{-1}=\;&a^{-1}z\;\! a \Longleftrightarrow a^2z=z\;\! a^2 \Longleftrightarrow a\;\! z=z\;\! a\;\! .
\end{split}
\end{equation*}
Concluding, the $*$-map $j_0$ enjoys the same properties as $j$, except the condition $j_0(M)= M'$.
\smallskip

As $j_0(e)$ is again an abelian projection with full central support, there exists a partial
isometry $v\in N$ such that $v^{*}v=e\;\!$, $vv^{*}=j_0(e)$. Define
\begin{equation}
\label{suun}
T:=j(a^{-1})\sum_{\iota\in I}j_0(v_{\iota})vJ_{e}v_{\iota}^{*}\,.
\end{equation}
It is standard matter to see that the above sum converges in the strong
operator topology and defines a bijective antilinear operator acting on $H$.
By taking into account the previous considerations and the fact that
$v_{\iota}^{*}xv_{\kappa}=z_{\iota \kappa}e$ with a uniquely determined element $z_{\iota \kappa}\in Z$
for any $\iota ,\kappa\in I$,
we get for  $x\in N$,
\medskip

\noindent\hspace{1.36 cm}$\displaystyle
TxT^{-1}=j(a^{-1})\bigg(\sum_{\iota,\kappa\in I}j_0(v_{\iota})vJ_{e}v_{\iota}^{*}xv_{\kappa}J_{e}
v^{*}j_0(v_{\kappa}^{*})\bigg)j(a)$

\noindent\hspace{2.7 cm}$\displaystyle =
j(a^{-1})\bigg(\sum_{\iota,\kappa\in I}j_0(v_{\iota})vJ_{e}z_{\iota \kappa}eJ_{e}v^{*}j_0(v_{\kappa}^{*})\bigg)j(a)$

\noindent\hspace{2.7 cm}$\displaystyle =
j(a^{-1})\bigg(\sum_{\iota,\kappa\in I}j_0(v_{\iota})vz_{\iota \kappa}^{*}ev^{*}j_0(v_{\kappa}^{*})\bigg)j(a)$

\noindent\hspace{2.7 cm}$\displaystyle =
j(a^{-1})\bigg(\sum_{\iota,\kappa\in I}j_0(v_{\kappa})z_{\iota \kappa}^{*}j_0(e)j_0(v_{\kappa}^{*})\bigg)j(a)$

\noindent\hspace{2.7 cm}$\displaystyle =
j(a^{-1})\bigg(\sum_{\iota,\kappa\in I}j_0(v_{\iota}z_{\iota \kappa}ev_{\kappa}^{*})\bigg)j(a)$

\noindent\hspace{2.7 cm}$\displaystyle =
j(a^{-1})\bigg(\sum_{\iota,\kappa\in I}j_0(v_{\iota}v_{\iota}^{*}xv_{\kappa}v_{\kappa}^{*})\bigg)j(a)$


\noindent\hspace{2.7 cm}$\displaystyle = j(a^{-1})j_0(x)j(a)=j(a^{-1})j(axa^{-1})j(a)=j(x)$
\medskip

\noindent as $j_0$ is a normal map. Consequently, the assumptions on $j$ in (iii)
imply that the operator $T$ defined in \eqref{suun} satisfies condition \eqref{thm-cond}.
\smallskip

(i)$\;\!\Rightarrow$(iv) As $M$ is assumed to be in standard form with conjugation $J$, by
Corollary \ref{prrci7} there exists an abelian projection $e\in N$ with central support
$z_N(e)=\idd_H$ and such that $e=JeJ$. Since
$$
JNJ=JZ'J=(JZJ)'=Z'=N\;\! ,
$$
we have
$$
J(Ne)J=(JNJ)(JeJ)=(JNJ)\;\! e=Ne\;\! .
$$
Therefore,
$$
T:Ne\ni x\mapsto JxJ\in Ne
$$
is an involutive $Z$-antilinear map. It is actually $Z$-antiunitary. Indeed, for every $x,y\in Ne$,
\begin{align*}
\F_e\big((JyJ)^*JxJ\big)e=\;&(JyJ)^*JxJ=Jy^*xJ=J\F_e(y^*x)eJ \\
=\;&\F_e(y^*x)^*e
\end{align*}
implies
$$
(Tx\!\mid\!Ty)_{\substack{ {} \\ Ne}}=\F_e\big((JyJ)^*JxJ\big)=\F_e(y^*x)^*=(x\!\mid\!y)^*_{\substack{ {} \\ Ne}}=(y\!\mid\!x)_{\substack{ {} \\ Ne}}\,.
$$

For $a\in N$ and $x\in Ne$, we compute
\begin{align*}
TL^{Ne}(a)T^{-1}x=\;&TL^{Ne}(a)JxJ=TaJxJ
=J(aJxJ)J=JaJx \\
=\;&L^{Ne}(JaJ)x
\end{align*}
which leads to
$$
TL^{Ne}(a)T^{-1}=L^{Ne}(JaJ)\,,\quad a\in N\,.
$$
Since $JMJ=M'$ and $JzJ=z^*$ for $z\in Z(M)\;\!$, we conclude that \eqref{czaza0} holds true.
\smallskip

(iv)$\;\!\Rightarrow$(v) Assuming (iv), let $e_0$ be an abelian projection of central support $\idd_H$ in $N$,
and $T_{e_0}$ a $Z$-antiunitary involution $Ne_0\longrightarrow Ne_0\;\!$, which satisfies
\eqref{czaza0}.

Consider now an arbitrary abelian projection $e\in N$ with $z_N(e)=\idd_H$, together with a
partial isometry $v\in N$ realising the equivalence $e\sim e_0\;\!$:
$$
v^*v=e_0\, ,\quad vv^*=e\, .
$$
By the choice of $e_0$ and $T_{e_0}\;\!$, the equality
$$
T_{e_0}L^{Ne_0}(a)T_{e_0} =L^{Ne_0}\big( \Phi (a)\big)
$$
defines a bijective, multiplicative, antilinear map $\Phi : M\longrightarrow M'$ such
that $\Phi (z)=z^*$ for all $z\in Z(M)\;\!$.
On the other hand, according to Proposition \ref{elleem}, the map 
$$
U_v:Ne\ni x\longmapsto xv\in Ne_0
$$
is $Z$-unitary with inverse $U_{v^*}$ and
$$
U_vL^{Ne}(a)=L^{Ne_0}(a)U_v\,,\quad a\in N\,.
$$
Therefore,
$$
T_e:=U_{v^*}T_{e_0}U_v : Ne\longrightarrow Ne
$$
is an involutive $Z$-antiunitary operator and it satisfies \eqref{czaza0} because, by the aboves,
we have for every $a\in M$
\begin{align*}
&T_eL^{Ne}(a)T_e^{-1}=U_{v^*}T_{e_0}U_vL^{Ne}(a)U_{v^*}T_{e_0}U_v\\
=\;&U_{v^*}T_{e_0}L^{Ne_0}(a)T_{e_0}U_v
=U_{v^*}L^{Ne_0}\big( \Phi (a)\big) U_v=L^{Ne}\big( \Phi (a)\big)\, .
\end{align*}


(v)$\;\!\Rightarrow$(vi) Let $e\in N$ be any abelian projection having central support
$z_N(e)=\idd_H\;\!$. According to (v), there exists a $Z$-antiunitary involution
$T_0 : Ne\longrightarrow Ne$ such that \eqref{czaza0} is satisfied.
Let $\Phi$ denote the bijective map $M\longrightarrow M'$ defined by
\begin{equation}\label{Phi}
T_0\;\! L^{Ne}(a)T_0 =L^{Ne}\big( \Phi (a)\big)\;\! ,\qquad a\in N\;\! .
\end{equation}

Since $T_0$ is $Z$-antiunitary, setting $v:=T_0\;\! e\in Ne\;\!$, we have
$$
v^*v=(v\!\mid\!v)_{Ne}\;\! e = (e\!\mid\!e)_{Ne}\;\! e = e\;\! ,
$$
so $v\in N$ is a partial isometry with $v^*v=e\;\!$. Therefore $f:=v v^*\in N$ is an
abelian projection with $z_N(f)=\idd_H$ and $v$ realizes the equivalence $e\sim f$:

\centerline{$v^*v=e\;\! ,\quad v v^*=f\, .$}
\smallskip

\noindent By Proposition \ref{elleem}, the map
\begin{equation*}
U_v : Nf\ni y\longmapsto y\;\! v\in Ne
\end{equation*}
is $Z$-unitary with inverse $U_{v^*}$ and
\begin{equation}\label{intertw}
U_v L^{Nf}(a) = L^{Ne}(a)U_v\;\! ,\quad a\in N\;\! .
\end{equation}

Put $T:=U_{v^*}T_0 : Ne\longrightarrow Nf\;\!$. It is clearly $Z$-antiunitary and
$$
T e=U_{v^*}T_0\;\! e =U_{v^*} v = v v^* = f\;\! .
$$
Furthermore, we have for every $a\in M$
\begin{equation*}
TL^{Ne}(a)T^{-1\! }=U_{v^*}T_0\;\! L^{Ne}(a)T_0\;\! U_v \!\overset{(\ref{Phi})}{=}\!
U_{v^*}L^{Ne}\big( \Phi (a)\big) U_v \!\overset{(\ref{intertw})}{=}\! L^{Nf}\big( \Phi (a)\big) ,
\end{equation*}
and this implies \eqref{czaza2}.
\smallskip

(vi)$\;\!\Rightarrow$(vii) is obvious.
\smallskip

We end the proof of the theorem by showing that (vii)$\;\!\Rightarrow$(i). We will
split this implication in (vii)$\;\!\Rightarrow$(viii) and (viii)$\;\!\Rightarrow$(i), where
(viii) is the following intermediate condition:
\smallskip

\begin{itemize}
\item[(viii)] For each projection $p\in Z(M)$
and for each abelian projection $g\in pNp$ with $z_{pNp}(g)=p\;\!$,
there exists an abelian projection $e_{p,g}\in N$ with $z_N(e_{p,g})=\idd_H$
such that
\[
e_{p,g}p=p\;\! e_{p,g}=g\;\! ,\quad pNe_{p,g}=(pNp)g\;\! ,
\]
as well as a bijective $Z$-antilinear map $T_{p,g} : Ne_{p,g}\longrightarrow Ne_{p,g}$
for which
\begin{equation}
\label{czaza1}
\begin{split}
T_{p,g}L^{Ne_{p,g}}(M){T_{p,g}}^{\! -1}=\;&L^{Ne_{p,g}}(M')\, ,\\
T_{p,g}L^{Ne_{p,g}}(z){T_{p,g}}^{\! -1}=\;&L^{Ne_{p,g}}(z^{*})\,,\quad z\in Z(M)
\end{split} 
\end{equation}
holds true.
\end{itemize}
\smallskip

First we prove implication (vii)$\;\!\Rightarrow$(viii).

We start by noticing that the equality
\smallskip

\centerline{$TL^{Ne}(a)T^{-1}=L^{Nf}(\Psi(a))\;\! ,\quad a\in M$}
\smallskip

\noindent defines a bijective, multiplicative antilinear map $\Psi : M\longrightarrow M'$
such that $\Psi(z)=z^*$ for each $z\in Z(M)\;\!$.
\smallskip

Let the projection $p\in Z(M)$ and the abelian projection $g\in pNp$ with
$z_{pNp}(g)=p$ be arbitrary. According to Theorem \ref{nopr1}, (ii), there
exists an abelian projection $e_{p,g}\in N$ with $z_N(e_{p,g})=\idd_H$
such that (among other properties)
\smallskip

\centerline{$e_{p,g}p=p\;\! e_{p,g}=f\, ,\; pNe_{p,g}=(pNp)g\;\! ,$}
\smallskip

\noindent Since the abelian projections $e_{p,g}\;\!$, $e\;\!$, $\! f$ are equivalent in $N$,
there are partial isometries $v\;\! ,w\in N$ such that
\begin{align*}
&v^* v=e\,,\quad v\;\! v^* =e_{p,g}\,,\\
&w^* w=e_{p,g}\,,\quad w\;\! w^* =f\,.
\end{align*}
By Proposition \ref{elleem}, the maps
\begin{align*}
&U_{v}:Ne_{p,g}\ni x\longmapsto xv\in Ne\,,\\
&U_{w}:Nf\ni y\longmapsto yw\in Ne_{p,g}
\end{align*}
are $Z$-unitaries and, for each $a\in N$,
\begin{align*}
&U_{v}L^{Ne_{p,g}}(a)=L^{Ne}(a)U_{v}\,,\\
&U_{w}L^{Nf}(a)=L^{Ne_{p,g}}(a)U_{w}
\end{align*}
hold true. Denoting
$$
T_{p,g}:=U_{w}TU_{v}:Ne_{p,g}\longrightarrow Ne_{p,g}\,,
$$
$T_{p,g}$ is a bijective $Z$-antilinear map. For each $a\in M$, we obtain 
\medskip

\noindent\hspace{3 cm}$T_{p,g} L^{Ne_{p,g}}(a)=U_{w}TU_{v}L^{Ne_{p,g}}(a)$
\smallskip

\noindent\hspace{5.41 cm}$=U_{w}TL^{Ne}(a)U_{v}$
\smallskip

\noindent\hspace{5.41 cm}$=U_{w}\big(TL^{Ne}(a)T^{-1}\big)TU_{v}$
\smallskip

\noindent\hspace{5.41 cm}$=U_{w}L^{Nf}\big(\Psi(a)\big)TU_{v}$
\smallskip

\noindent\hspace{5.41 cm}$=L^{Ne_{p,g}}\big(\Psi(a)\big)U_{w}TU_{v}$
\smallskip

\noindent\hspace{5.41 cm}$=L^{Ne_{p,g}}\big(\Psi(a)\big)T_{p,g}\,,$
\smallskip

\noindent that is
$$
T_{p,g} L^{Ne_{p,g}}(a)T_{p,g}^{\! -1}=L^{Ne_{p,g}}(\Psi(a))\,.
$$
In particular, (\ref{czaza1}) holds true.
\smallskip

Finally we prove also implication (viii)$\;\!\Rightarrow$(i).

Let us decompose $M$ in its properly infinite and finite part: let $p_0\in Z(M)$ be
the projection such that $Mp_0$ is properly infinite and $M(\idd_H-p_0)$ finite.

Since the finite projection $\idd_H-p_0$ is piecewise of countable type (see  \cite{SZ1}, Lemma 7.2), we can further decompose $M(\idd_H-p_0)$ in
portions of countable type: there exists a family $(p_\iota)_{\iota\in I}$
of mutually orthogonal central projections in $Z(M)$ such that $\sum\limits_{\iota\in I}p_\iota=\idd_H-p_0$
and all von Neumann algebras $M_{p_\iota}\subset B(p_\iota H)$ are of countable type.

Finally, according to Lemma 7.18 of \cite{SZ1}, each $p_\iota$ is the (orthogonal) sum of
two projections $p^{(1)}_\iota, p^{(2)}_\iota\in Z(M)$ such that $M_{p^{(1)}_\iota}$ and 
$\big(M_{p^{(2)}_\iota}\big)'=(M')_{p^{(2)}_\iota}$ are cyclic von Neumann algebras.

Summarising,
$$
p_0+\sum_{\iota\in I}p^{(1)}_\iota+\sum_{\iota\in I}p^{(2)}_\iota=\idd_H
$$
where $M_{p_0}$ is properly infinite, $M_{p^{(1)}_\iota}$ and $(M')_{p^{(2)}_\iota}$ are cyclic, $\iota\in I$.

Thus, denoting $\G:=\{0\}\bigcup\big(\{1,2\}\times I\big)$
and
\begin{align*}
&p_\g=p_0\,\,\text{for}\,\,\g=0\,,\\
&p_\g=p_\iota^{(j)}\,\,\text{for}\,\,\g=(j,\iota)\,,\,\,\,\,j=1,2\,,\,\iota\in I\,,
\end{align*}
we have a family $(p_\g)_{\g\in\G}$ of mutually orthogonal projections in $Z(M)$ such that
$\sum_{\g\in\G}p_\g=\idd_H$ and, for each $\g\in\G$,
\begin{itemize}
\item[-] either $M_{p_\g}$ is properly infinite,
\item[-] or $M_{p_\g}$ admits a cyclic vector $\xi_\g\in p_\g H$,
\item[-] or $(M')_{p_\g}$ admits a cyclic vector $\xi_\g\in p_\g H$.
\end{itemize}
By showing that, for each $\g\in\G$, the reduced/induced von Neumann algebra
$M_{p_\g}$ is in standard form, it will follow that $M$ is in standard form.
\smallskip

Let $\g\in\G$ be arbitrary, and let $g_\g$ denote
\begin{itemize}
\item[-] any abelian projection in $p_\g Np_\g$ with $z_{p_\g Np_\g}(g_\g)=p_\g$
if $M_{p_\g}$ is properly infinite,
\item[-] the abelian projection in $p_\g Np_\g$ satisfying $g_\g\xi_\g =\xi_\g\;\!$,
whose existence and uniqueness is guaranteed by Proposition \ref{prrci} and which
necessarily satisfies $z_{p_\g Np_\g}(g_\g)=p_\g\;\!$, if $M_{p_\g}$ or $(M')_{p_\g}$
admits a cyclic vector $\xi_\g\in p_\g H$.
\end{itemize}
By the assumed condition (viii), there exist an abelian projection $e_\g\in N$ with
$z_N(e_\g)=\idd_H$ such that
\begin{equation}
\label{sssste}
e_\g p_\g =p_\g e_\g =g_\g\;\! ,\; p_\g Ne_\g =(p_\g Np_\g )g_\g\,,
\end{equation}
as well as a bijective $Z$-antilinear map $T_\g : Ne_\g\longrightarrow
Ne_\g$ for which
\begin{equation}
\label{czaza1bis}
\begin{split}
& T_\g L^{Ne_\g}(M){T_\g}^{\! -1}=L^{Ne_\g}(M')\, ,\\
& T_\g L^{Ne_\g}(z){T_\g}^{\! -1}=L^{Ne_\g}(z^{*})\,,\quad z\in Z(M)
\end{split} 
\end{equation}
holds true. We notice that the equality
\begin{equation}
\label{implement}
T_\g L^{Ne_\g}(a){T_\g}^{\! -1}=L^{Ne_\g}(\Psi_\g (a))\, ,\quad a\in M
\end{equation}
defines a bijective, multiplicative antilinear map $\Psi_\g : M\longrightarrow M'$
such that $\Psi_\g (z)=z^*$ for each $z\in Z(M)\;\!$.

According to the second equality in (\ref{czaza1bis}), we have for each projection
$p\in Z(M)$,
\smallskip

\centerline{$T_\g L^{Ne_\g}(p)=\big( T_\g L^{Ne_\g}(p){T_\g}^{\! -1}\big) T_\g
=L^{Ne_\g}(p)T_\g\;\! ,$}
\smallskip

\noindent in particular $T_\g$ and $L^{Ne}(p_\g )$ commute.

Consequently, by  (\ref{implement}) and the first equality in (\ref{czaza1bis}),
\medskip

\noindent\hspace{0.46 cm}$L^{Ne_\g}(\Psi_\g (Mp_\g ))=T_\g L^{Ne_\g}(Mp_\g ){T_\g}^{\! -1}$
\smallskip

\noindent\hspace{3.45 cm}$=T_\g L^{Ne_\g}(M)L^{Ne_\g}(p_\g){T_\g}^{\! -1}$
\smallskip

\noindent\hspace{3.45 cm}$=T_\g L^{Ne_\g}(M){T_\g}^{\! -1}L^{Ne_\g}(p_\g)
=L^{Ne_\g}(M')L^{Ne_\g}(p_\g)$
\smallskip

\noindent\hspace{3.45 cm}$=L^{Ne_\g}(M'p_\g)\,,$
\smallskip

\noindent and therefore the restriction of $\Psi_\g$ to $Mp_\g$ is a bijective,
multiplicative antilinear map onto $M'p_\g\;\!$, which acts on the common
centre $Z(M)p_\g$ clearly as the $*$-operation. It follows that, associating
to every $a\lceil_{p_\g H}$ with $a\in M$ the operator
$\Psi_\g (a)\lceil_{p_\g H}\;\!$, we get a bijective, multiplicative antilinear map
of the von Neumann algebra $M_{p_\g}\subset B(p_\g H)$ onto its commutant
$(M_{p_\g})'=(M')_{p_\g}\;\!$, which acts on the centre $Z(M)_{p_\g}$ as the
$*$-operation.

Thus, if the reduced/induced von Neumann algebra $M_{p_\g}$ is properly
infinite, Proposition 2.3 of \cite{FZ} yields its standardness.

It remains to show the standardness of $M_{p_\g}$ also in the case when
$M_{p_\g}$ or $(M')_{p_\g}$ admits a cyclic vector $\xi_\g\in p_\g H$.
We recall that in this case the abelian projection $g_{\g}$ in $p_{\g}Np_{\g}$
was chosen such that $g_{\g}\xi_\g =\xi_\g\;\!$.

Using again the commutation of $T_\g$ and $L^{Ne}(p_\g )\;\!$, we obtain for
every $x\in N$
\smallskip

\centerline{$T_\g (p_\g x e_\g) =\big( T_\g L^{Ne_{\g}}(p_\g )\big) (x e_\g )
=\big( L^{Ne_{\g}}(p_\g )T_\g \big) (x e_\g )=p_\g T_\g (x e_\g )$}
\smallskip

\noindent\hspace{2.3 cm}$\in p_\g Ne_\g\, ,$
\smallskip

\noindent that is $T_\g$ leaves invariant $p_\g Ne_\g\stackrel{\eqref{sssste}}{=}
(p_\g Np_\g )g_\g\;\! .$
Similarly, also ${T_\g}^{\! -1}$ leaves $(p_\g Np_\g )g_\g$ invariant, so the
restriction $T_\g\lceil_{(p_\g Np_\g )g_\g}$ of $T_\g$ to $(p_\g Np_\g )g_\g$
is a bijection $\widetilde{T_\g}:= T_\g\lceil_{(p_\g Np_\g )g_\g} : (p_\g Np_\g )g_\g
\longrightarrow (p_\g Np_\g )g_\g\;\!$.

Using (\ref{implement}), we obtain for every $a\in p_\g Mp_\g$
and $x\in (p_\g Np_\g )g_\g$
\smallskip

\centerline{$\widetilde{T_\g} (ax)=\big( T_{\g}L^{Ne_{\g}}(a)\big) x =
\Big( L^{Ne_{\g}}\big(\Psi_{\g}(a)\big) T_{\g}\Big) x =
\Psi_{\g}(a) (\widetilde{T_{\g}}x)\, ,$}
\smallskip

\noindent hence
\smallskip

\centerline{$\widetilde{T_{\g}} L^{(p_\g Mp_\g )g_\g} (a) =
L^{(p_\g Mp_\g )g_\g} \big(\Psi_{\g}(a)\big) \widetilde{T_{\g}}\, ,\quad
a\in p_\g Mp_\g\, .$}
\smallskip

\noindent Now we can apply Proposition \ref{monst} to
\begin{itemize}
\item[-] the reduced/induced von Neumann algebra $M_{p_{\g}}\subset B(p_{\g}H)\;\!$,
\item[-] the cyclic vector $\xi_{\g}\;\!$,
\item[-] the abelian projection  $g_{\g}\lceil_{p_\g H}$ in $N_{p_{\g}}\;\!$,
\item[-] the bijective $Z$-antilinear map $N_{p_{\g}}(g_{\g}\lceil_{p_\g H})\longrightarrow
N_{p_{\g}}(g_{\g}\lceil_{p_\g H})$ induced by $\widetilde{T_{\g}}$
\end{itemize}
and deduce that the von Neumann algebra $M_{\g}$ is in standard form.

\end{proof}


Theorem \ref{st1} has a significant application in the theory of tensor products of $W^*$-algebras over a common $W^*$-subalgebra.

More precisely, let $M_1$, $M_2$, $Z$ be $W^*$-algebras with $Z$ abelian such that $Z$ is identified with $W^*$-subalgebras of the centres $Z(M_1)$, $Z(M_2)$ via some faithful normal $*$-homomorphisms $\iota_j:Z\rightarrow Z(M_j)$, $j=1,2$.

A pair $(\pi_1,\pi_2)$ of spatial representations $\pi_j:M_j\rightarrow \cb(H)$, $j=1,2$, is called a {\it splitting representation} of the quintuple $\big(M_1,M_2,Z,\iota_1,\iota_2\big)$ whenever
\begin{itemize}
\item[1)] $\pi_1\circ\iota_1=\pi_2\circ\iota_2$;
\item[2)] there exists a type $\ty{I}$ von Neumann algebra $N\subset\cb(H)$ satisfying
\begin{align*}
(\pi_j\circ\iota_j)(Z)=&Z(N)\,,\,\,j=1,2\,,\\
\pi_1(M_1)\subset& N\,,\,\,\pi_2(M_2)\subset N'\,.
\end{align*}
\end{itemize}
Roughly speaking, the images $\pi_1(M_1)$ and $\pi_2(M_2)$ can be separated by a type $\ty{I}$ von Neumann algebra whose centre coincides with the identified subalgebras $(\pi_1\circ\iota_1)(Z)$, $(\pi_2\circ\iota_2)(Z)$.

Splitting representations $(\pi_1,\pi_2)$ exist for any quintuple 
$\big(M_1,M_2,Z,\iota_1,\iota_2\big)$ as above (\cite{SZ}, Lemma 5.2). In addition, if $(\r_1,\r_2)$
is another splitting representation of the same quintuple, then there exists a (unique) $*$-isomorphism
$\pi_1(M_1)\bigvee\pi_2(M_2)\rightarrow \r_1(M_1)\bigvee\r_2(M_2)$ sending $\pi_1(M_1)$ in $\r_1(M_1)$
and $\pi_2(M_2)$ in $\r_2(M_2)$ (\cite{SZ}, Lemma 5.4).
Thus the $W^*$-algebra $\pi_1(M_1)\bigvee\pi_2(M_2)$ can be called (as in \cite{SZ}) the {\it $W^*$-tensor
product of $M_1$ and $M_2$ over the common $W^*$-subalgebra $Z$}. Notice that, for $Z=\bc$, the above
construction reduces to the usual tensor product $M_1\ots M_2$.

It can be proved that if $(\pi_1,\pi_2)$ is a splitting representation of the quintuple
$\big(M_1,M_2,Z,\iota_1,\iota_2\big)$ as above, and $\pi$ is a standard representation of the $W^*$-algebra 
$\pi_1(M_1)\bigvee\pi_2(M_2)$, then $(\pi\circ\pi_1,\pi\circ\pi_2)$ is still a splitting representation of
$\big(M_1,M_2,Z,\iota_1,\iota_2\big)$. The rather involved proof of such a result, provided in the
forthcoming paper \cite{FZ1}, heavily relies on the above Theorem \ref{st1}.
\smallskip

{\bf Acknowledgement.}
We are indebted to the referee for a careful reading of the manuscript.
His suggestions contributed to make the exposition more precise and fluent.

\bigskip\bigskip

\end{document}